\documentclass{amsart}
\usepackage[english]{babel}

\usepackage[a4paper,top=2cm,bottom=2cm,left=3cm,right=3cm,marginparwidth=1.75cm]{geometry}

\usepackage{amsmath}
\usepackage{graphicx}
\usepackage[colorlinks=true, allcolors=blue]{hyperref}
\usepackage{hyperref}
\usepackage{cite}

 \newtheorem{thm}{Theorem}[section]

 \newtheorem{prop}[thm]{Proposition}
\theoremstyle{definition}

 \theoremstyle{remark}

\DeclareMathOperator*{\ef}{\operatorname{ess\,inf}}
\DeclareMathOperator*{\essinf}{\operatorname{ess\,inf}}

\numberwithin{equation}{section}

\title[Harmonic analysis operators in the rational Dunkl settings]{Harmonic analysis operators in the rational Dunkl settings}
\author{V. Almeida, J.J. Betancor, J.C. Fari\~na and L. Rodr\'{\i}guez-Mesa}
\address{V\'{\i}ctor Almeida, Jorge J. Betancor, Juan C. Fari\~na and Lourdes Rodr\'{\i}guez-Mesa\newline
	Departamento de An\'alisis Matem\'atico, Universidad de La Laguna,\newline
	Campus de Anchieta, Avda. Astrof\'isico S\'anchez, s/n,\newline
	38721 La Laguna (Sta. Cruz de Tenerife), Spain}
\email{valmeida@ull.edu.es, jbetanco@ull.es, jcfarina@ull.edu.es,
lrguez@ull.edu.es
}

\thanks{The authors are partially supported by grant PID2019-106093GB-I00 from the Spanish Government}


\begin{document}
\maketitle

\begin{abstract}
In this paper we study harmonic analysis operators in Dunkl settings associated with finite reflection groups on Euclidean spaces. We consider maximal operators, Littlewood-Paley functions, $\rho$-variation and oscillation operators involving time derivatives of the heat semigroup generated by Dunkl operators. We establish the boundedness properties of these operators in $L^p(\mathbb R^d,\omega_\mathfrak{K})$, $1\leq p<\infty$, Hardy spaces, BMO and BLO-type spaces in the Dunkl settings. The study of harmonic analysis operators associated to reflection groups need  different strategies from the ones used in the Euclidean case since the integral kernels of the operators admit estimations involving two different metrics, namely, the Euclidean and the orbit metrics. For instance, the classical Calder\'on-Zygmund theory for singular integrals does not work in this setting.
\end{abstract}

\section{Introduction}\label{S1}
In this paper we study harmonic analysis operators in the Dunkl setting which is associated with finite reflection groups in Euclidean spaces. The Dunkl theory can be seen as an extension of the Euclidean Fourier analysis. After the Dunkl's paper ($\!\!$\cite{Dun1}) appeared the theory has been developed extensively ($\!\!$\cite{DaXu}, \cite{Dun2}, \cite{Dun3}, \cite{Dun4}, \cite{Ro1,Ro2,Ro4,Ro3,JeuRo,RoVo}, \cite{ThXu1} and \cite{ThXu2}). Recently, harmonic analysis associated with Dunkl operators has gained a lot of interest ($\!\!$\cite{AS, ABDH, ADH}, \cite{DH2,DH1,DH4,DH3,DH5,DH6,DH7,DH8}, \cite{HLLW}, \cite{H1},  \cite{JiLi1,JiLi3,JiLi2},  \cite{LZ} and \cite{THHLL}).

We now collect some basic definitions and properties concerning to Dunkl theory that will be useful in the sequel. The interested reader can go to \cite{Dun1} and \cite{Ro4} for details.

We consider the Euclidean space $\mathbb R^d$ endowed with the usual inner product $\langle\cdot ,\cdot\rangle$. If $\alpha$ is a nonzero vector in $\mathbb R^d$ the reflection $\sigma_\alpha$ with respect to the hyperplane orthogonal to $\alpha$ is defined by

$$\sigma_\alpha(x)=x-2\frac{\langle x,\alpha\rangle}{|\alpha|^2}\alpha,\quad x\in\mathbb R^d.$$
Here $|\alpha|$ denotes as usual the Euclidean norm of $\alpha$. Let $R$ be a finite set of $\mathbb R^d\setminus\{0\}$, we say that $R$ is a root system if $\sigma_\alpha(R)=R$ and $R\cap(\mathbb R\alpha)=\{\pm\alpha\}$, $\alpha\in R$. 
The reflection $\{\sigma_\alpha\}_{\alpha\in R}$ generates a finite group G called Weyl group of the root system $R$. We always consider that the system of root is normalised being $|\alpha|=\sqrt{2}$, $\alpha\in R$. The root system $R$ can be rewritten as a disjoint union $R=R_+\cup(-R_+)$, where $R_+$ and $-R_+$ are separated by a hyperplane through the origin. $R_+$ is named a positive subsystem of $R$. This decomposition of $R$ is not unique.

If $A$ is a subset of $\mathbb R^d$, by $\theta(A)$ we denote the $G$-orbit of $A$. We define $\rho(x,y)=\min_{\sigma\in G}|x-\sigma(y)|$, $x,y\in\mathbb R^d$, that represents the distance between the orbits $\theta(\{x\})$ and $\theta(\{y\})$. For every $x\in\mathbb R^d$ and $r>0$ we denote by $B(x,r)$ and $B_\rho(x,r)$ the Euclidean and the $\rho$-balls centered in $x$ and with radius $r$. We have that $B_\rho(x,r)=\theta(B(x,r))$, $x\in\mathbb R^d$ and $r>0$.

A multiplicity function is a $G$-invariant function $\mathfrak{K}:R\rightarrow\mathbb C$. Throughout this paper we always consider a multiplicity function $\mathfrak{K}\geq 0$. We define $\gamma =\sum_{\alpha\in R_+}\mathfrak{K}(\alpha)$ and $D=d+2\gamma$. The measure $\omega_\mathfrak{K}$ defined by $d\omega_\mathfrak{K}(x)=\omega_\mathfrak{K}(x)dx$, on $\mathbb R^d$, where
$$
\omega_\mathfrak{K}(x)=\prod_{\alpha\in R_+}|\langle\alpha,x\rangle|^{2\mathfrak{K}(\alpha)},\quad x\in\mathbb R^d,$$
is $G$-invariant. The number $D$ is called the homogeneous dimension because the following property holds
$$
\omega_\mathfrak{K}(B(tx,tr))=t^D\omega_\mathfrak{K}(B(x,r)),\quad x\in\mathbb R^d\;\mbox{and}\;r>0.
$$
We have that 
$$
\omega_\mathfrak{K}(B(x,r))\sim r^d\prod_{\alpha\in R}(|\langle\alpha,x\rangle|+r)^{\mathfrak{K}(\alpha)},\quad x\in\mathbb R^d\;\mbox{and}\;r>0,$$
where the equivalence constants do not depend on $x\in\mathbb R^d$ or $r>0$. Then, we can see that $\omega_\mathfrak{K}$ has the following doubling property: there exists $C>0$ such that
$$\omega_\mathfrak{K}(B(x,2r))\leq C\omega_\mathfrak{K}(B(x,r)),\quad x\in\mathbb R^d\;\mbox{and}\;r>0.$$
Hence the triple $(\mathbb R^d,|\cdot|,\omega_\mathfrak{K})$ is a space of homogeneous type in the sense of Coifman and Weiss ($\!\!$\cite{CoWe}).

A very useful property ($\!\!$\cite[(3.2)]{ADH}) is the following one: there exists $C\geq 1$ such that
\begin{align}\label{1.1}
  \frac{1}{C}\left(\frac{s}{r}\right)^d \leq \frac{\omega_\mathfrak{K}(B(x,s))}{\omega_\mathfrak{K}(B(x,r))}\leq C \left(\frac{s}{r}\right)^D,\quad x\in\mathbb R^d\;\mbox{and}\;0<r\leq s.
\end{align}
Also, we have that
\begin{align}\label{1.2}
  \omega_\mathfrak{K}(B(x,r))\leq \omega_\mathfrak{K}(\theta (B(x,r)))\leq \#(G) \omega_\mathfrak{K}(B(x,r)),\quad x\in\mathbb R^d\;\mbox{and}\;r>0.
\end{align}
Here $\#(G)$ denotes the number of elements of $G$.

Some examples of these objects can be found in \cite[p. 96]{Ro4}.

We now give the definitions of Dunkl operators associated with a system of roots $R$ and a multiplicity function $\mathfrak K$. Let $\xi\in\mathbb R^d\setminus\{0\}$. We denote by $\partial_\xi$ the classical derivative in the direction $\xi$. The Dunkl operator $D_\xi$ is defined by

$$D_\xi f(x)=\partial_\xi f(x)+\sum_{\alpha\in R}\frac{\mathfrak{K}(\alpha)}{2}\langle\alpha,\xi\rangle\frac{f(x)-f(\sigma_\alpha(x))}{\langle\alpha,x\rangle},\quad x\in\mathbb R^d.$$
$D_\xi$ can be seen as a deformation of $\partial_\xi$ by a difference operator. These operators were introduced in \cite{Dun1} where their main properties can be found. The Dunkl Laplacian associated with $R$ and $\mathfrak{K}$ is defined by $\Delta_\mathfrak K=\sum_{j=1}^dD_{e_j}^2$, where $\{e_j\}_{j=1}^d$ represents the canonical basis in $\mathbb R^d$. It is clear that $\Delta_\mathfrak{K}$ reduces to the Euclidean Laplacian when $\mathfrak K=0$. We can write, for every $f\in C^2(\mathbb R^d)$,
$$
\Delta_\mathfrak K f=\Delta f+\sum_{\alpha\in R}\mathfrak{K}(\alpha)\delta_\alpha f,
$$
where $\Delta$ denotes the Euclidean Laplacian and, for every $\alpha\in R$,
$$
\delta_\alpha f(x)=\frac{\partial_\alpha f(x)}{\langle\alpha,x\rangle}-\frac{f(x)-f(\sigma_\alpha(x))}{\langle\alpha,x\rangle^2},\quad x\in\mathbb R^d.
$$
According to \cite[Theorem 3.1]{AmHa} the Dunkl Laplacian is essentially selfadjoint in $L^2(\mathbb R^d,\omega_\mathfrak{K})$ and $\Delta_\mathfrak K$ generates a $C_0$-semigroup of operators $\{T_t^\mathfrak K\}_{t>0}$ in $L^2(\mathbb R^d,\omega_\mathfrak{K})$ where, for every $t>0$ and $f\in L^2(\mathbb R^d,\omega_\mathfrak{K})$,
\begin{align}\label{1.3}
  T_t^\mathfrak K(f)(x)=\int_{\mathbb R^d}T_t^\mathfrak K (x,y)f(y)d\omega_\mathfrak{K}(y),
\end{align}
being the function $(t,x,y)\rightarrow T_t^\mathfrak K (x,y)$ a $C^\infty$-function on $(0,\infty)\times\mathbb R^d\times\mathbb R^d$ satisfying that $T_t^\mathfrak K (x,y)=T_t^\mathfrak K (y,x)$, $t>0$ and $x,y\in\mathbb R^d$, and

\begin{align}\label{1.4}
 \int_{\mathbb R^d}T_t^\mathfrak K (x,y)d\omega_\mathfrak{K}(y)=1,\quad x\in\mathbb R^d\;\mbox{and}\;t>0.
\end{align}
The integral operator given by \eqref{1.3} defines a $C_0$-semigroup of linear contractions in $L^p(\mathbb R^d,\omega_\mathfrak{K})$, $1\leq p<\infty$. Thus $\{T_t^\mathfrak K\}_{t>0}$ defines a symetric diffusion semigroup in the sense of Stein ($\!\!$\cite{St}).

Gaussian type upper bounds were established in \cite[Theorem 4.1]{ADH} for the heat kernel and its derivatives. We now recall those estimates that we will use throughout this paper.
\begin{itemize}
\item[(a)] For every $m\in\mathbb N$ there exist $C,c>0$ such that
\begin{align}\label{1.5}
 |t^m\partial _t^mT_t^\mathfrak K (x,y)|\leq \frac{C}{V(x,y,\sqrt{t})}e^{-c\frac{\rho(x,y)^2}{t}},\quad x,y\in\mathbb R^d\;\mbox{and}\;t>0.
\end{align}
Here $V(x,y,r)=\max\{\omega_\mathfrak{K}(B(x,r)),\omega_\mathfrak{K}(B(y,r))\}$, $x,y\in\mathbb R^d$ and $r>0$.
\item[(b)] For every $m\in\mathbb N$ there exist $C,c>0$ such that
\begin{align}\label{1.6}
 |t^m\partial _t^mT_t^\mathfrak K (x,y)-t^m\partial _t^mT_t^\mathfrak K (x,z)|\leq C\frac{|y-z|}{\sqrt{t}}\frac{e^{-c\frac{\rho(x,y)^2}{t}}}{V(x,y,\sqrt{t})},
\end{align}
for every $t>0$, $x,y,z\in\mathbb R^d$ and $|y-z|<\sqrt{t}$.
\end{itemize}

Note that in the exponential term in (\ref{1.5}) and (\ref{1.6}) the Euclidean metric is replaced by the orbit distance $\rho$. Moreover in the estimate (\ref{1.6}) appear the two metrics. Note that these two metrics are not equivalent.

The fact that estimates involve the two metrics can not be avoid when we are working in the Dunkl setting as it can be seen in the Dunkl-Calder\'on Zygmund singular integrals (see \cite{HLLW}). These properties differ in Dunkl and Euclidean settings, so we need more careful arguments when we are dealing in Dunkl context.

Our objective in this paper is to study the behaviour of some Dunkl harmonic analysis operators involving the heat semigroup $\{T_t^\mathfrak K\}_{t>0}$ and its time derivatives in $L^p$, Hardy and ${\rm BMO}$ spaces.

In addition to Lebesgue spaces $L^p(\mathbb R^d, \omega_\mathfrak{K})$, $1\leq p<\infty$, and the weak $L^1$-space denoted by $L^{1,\infty}(\mathbb R^d, \omega_\mathfrak{K})$, we also think about the Hardy space $H^1(\Delta_\mathfrak{K})$ and the ${\rm BMO}$ type spaces ${\rm BMO}(\mathbb R^d, \omega_\mathfrak{K})$, ${\rm BMO}^\rho (\mathbb R^d,\omega_\mathfrak K)$ and ${\rm BLO}(\mathbb R^d, \omega_\mathfrak{K})$ which we are going to precise.

We consider Hardy spaces associated with Dunkl operators introduced in \cite{ADH} and \cite{DH1}.

Let $1<q\leq\infty$ and let $M$ be a positive integer. We say that a function $a$ is a $(1,q,\Delta_\mathfrak K ,M)$-atom when $a\in L^2(\mathbb R^d,\omega_\mathfrak{K})$ and there exist $b\in D(\Delta_\mathfrak K ^M)$ and an Euclidean ball $B=B(x_B,r_B)$ with $x_B\in\mathbb R^d$ and $r_B>0$, satisfying that

\begin{enumerate}
\item[(i)] $a=\Delta_\mathfrak K ^M b$;
\item[(ii)] $\mbox{supp} (\Delta_\mathfrak K ^\ell b)\subset\theta(B)$, $\ell=0,...,M$;
\item[(iii)] $\|(r_B^2\Delta_\mathfrak K )^\ell b\|_{L^q(\mathbb R^d,\omega_\mathfrak{K})}\leq r_B^{2M}\omega_\mathfrak{K}(B)^{1/q-1}$, $\ell=0,...,M$.
\end{enumerate}
A function $f$ is in $H^1_{(1,q,\Delta_\mathfrak K ,M)}$ when $f=\sum_{j\in \mathbb N}\lambda_j a_j$, where, for every $j\in\mathbb N$, $a_j$ is a $(1,q,\Delta_\mathfrak K ,M)$-atom and $\lambda_j\in\mathbb C$ such that $\sum_{j\in \mathbb N} |\lambda_j|<\infty$. Here the series defining $f$ converges in $L^1(\mathbb R^d,\omega_\mathfrak{K})$. We define, for every $f\in H^1_{(1,q,\Delta_\mathfrak K ,M)}$

$$\|f\|_{H^1_{(1,q,\Delta_\mathfrak K ,M)}}=\inf \sum_{j\in \mathbb N} |\lambda_j|,$$
where the infimum is taken over all the sequences $\{\lambda_j\}_{j\in \mathbb N}\subset\mathbb C$ such that $\sum_{j\in \mathbb N}|\lambda_j|<\infty$ and $f=\sum_{j\in \mathbb N} \lambda_j a_j$, where, for every $j\in\mathbb N$, $a_j$ is a $(1,q,\Delta_\mathfrak K,M)$-atom.

Let $1<q\leq\infty$. A function $a$ is said to be a $(1,q)$-atom if there exists an Euclidean ball $B$ such that

\begin{enumerate}
\item[(i)] $\mbox{supp}\;a\subset B$;
\item[(ii)] $\int ad\omega_\mathfrak{K}=0$;
\item[(iii)] $\|a\|_{L^q(\mathbb R^d,\omega_\mathfrak{K})}\leq \omega_\mathfrak{K}(B)^{1/q-1}$.
\end{enumerate}
The Hardy space $H^1_{(1,q)}$ is defined as above replacing $(1,q,\Delta_\mathfrak K ,M)$-atoms by $(1,q)$-atoms.
In \cite[Theorem 1.5]{DH1} it was proved that $H^1_{(1,q)}=H^1_{(1,q,\Delta_\mathfrak K ,M)}$ algebraically and topologically. The space $H^1_{(1,q)}$ is characterized by using maximal functions ($\!\!$\cite[Theorem 2.2]{ADH}), square functions ($\!\!$\cite[Theorem 2.3]{ADH}) and Riesz transforms ($\!\!$\cite[Theorem 2.5]{ADH}), what makes clear that $H^1_{(1,q)}$ does not depend on $q$. From now on we denote by $H^1(\Delta_\mathfrak K)$ any of these Hardy spaces.

According to the results in \cite{CoWe1} about Hardy spaces in homogeneous type spaces the dual of $H^1(\Delta _\mathfrak K)$ can be characterized as the space of bounded mean oscillation functions ${\rm BMO}(\mathbb R^d,\omega_\mathfrak{K})$ defined as follows. A function $f\in L^1_{\rm loc}(\mathbb R^d,\omega_\mathfrak{K})$ is in ${\rm BMO}(\mathbb R^d,\omega_\mathfrak{K})$ provided that
$$
\|f\|_{{\rm BMO}(\mathbb R^d,\omega_\mathfrak{K})}:=\sup_{B}\frac{1}{\omega_\mathfrak{K}(B)}\int_B |f(y)-f_B|d\omega_\mathfrak{K}(y)<\infty.
$$
Here the supremum is taken over all the Euclidean balls $B$ in $\mathbb R^d$. For every Euclidean ball $B$ we define $f_B=\frac{1}{\omega_\mathfrak{K}(B)}\int_B fd\omega_\mathfrak{K}$. As it is well-known $({\rm BMO}(\mathbb R^d,\omega_\mathfrak{K}), \|\cdot\|_{{\rm BMO}(\mathbb R^d,\omega_\mathfrak{K})})$ is a Banach space when functions differing in a constant are identified. In \cite[Theorem 6.7]{JiLi3} it was proved that the dual of $H^1(\Delta _\mathfrak K)$ can be also realized by a class of functions defined by using Carleson measures in the Dunkl setting.

Other ${\rm BMO}$-type space in the Dunkl setting can be considered by replacing the Euclidean balls by $\rho$-metric balls. The space ${\rm BMO}^\rho(\mathbb R^d,\omega_\mathfrak{K})$ consists of all those $f\in L^1_{\rm loc}(\mathbb R^d,\omega_\mathfrak{K})$ such that 
$$
\|f\|_{{\rm BMO}^\rho(\mathbb R^d,\omega_\mathfrak{K})}:=\sup_{B}\frac{1}{\omega_\mathfrak{K}(\theta(B))}\int_{\theta(B)} |f(y)-f_{\theta(B)}|d\omega_\mathfrak{K}(y)<\infty,
$$
where the supremum is taken over all the Euclidean balls $B$ in $\mathbb R^d$. We have that ${\rm BMO}^\rho(\mathbb R^d,\omega_\mathfrak{K})$ is contained in ${\rm BMO}(\mathbb R^d,\omega_\mathfrak{K})$ and $f\in {\rm BMO}^\rho(\mathbb R^d,\omega_\mathfrak{K})$ provided that $f$ is $G$-invariant and $f\in {\rm BMO}(\mathbb R^d,\omega_\mathfrak{K})$ ($\!\!$\cite[Proposition 7.4]{JiLi3}). As it is proved in \cite[Section 7.2]{JiLi3}  ${\rm BMO}^\rho(\mathbb R^d,\omega_\mathfrak{K})$ does not coincide with  ${\rm BMO}(\mathbb R^d,\omega_\mathfrak{K})$. Furthermore,  Han, Lee, Li and Wick ($\!\!$\cite{HLLW}) analyzed the $L^p$-boundedness of the commutator of the Dunkl-Riesz transform with functions $b$ on these ${\rm BMO}$-type spaces in the Dunkl setting. They established that the commutator is bounded on $L^p(\mathbb R^d,\omega _\mathfrak K)$, $1<p<\infty$, when the function $b$ belongs to ${\rm BMO}^\rho (\mathbb R^d,\omega_\mathfrak{K})$. Conversely, if the commutator of the Dunkl-Riesz transform with $b$ is bounded on $L^p(\mathbb R^d,\omega _\mathfrak K)$ for some $1<p<\infty$, then $b\in {\rm BMO}(\mathbb R^d,\omega_\mathfrak K)$ ($\!\!$\cite[Theorem 1.3]{HLLW}).

Coifman and Rochberg ($\!\!$\cite{CR}) introduced the space ${\rm BLO}$ of functions of bounded lower oscillation in the Euclidean setting (see also \cite{Ben}). ${\rm BLO}$ space is defined analogously of ${\rm BMO}$ but replacing the average $f_B$ by the essential infimum of $f$ in $B$. ${\rm BLO}$-type spaces appear as the image of $L^\infty$ and ${\rm BMO}$-spaces for maximal operators, Littlewood-Paley functions and singular integrals ($\!\!$\cite{Ji}, \cite {MY} and \cite{YYZ}).

We now introduce a ${\rm BLO}$-type space in the Dunkl setting. We say that a function $f\in L^1_{\rm loc}(\mathbb R^d,\omega_\mathfrak{K})$ is in ${\rm BLO}(\mathbb R^d,\omega_\mathfrak{K})$ when
$$
\|f\|_{{\rm BLO}(\mathbb R^d,\omega_\mathfrak{K})}:=\sup_{B}\frac{1}{\omega_\mathfrak{K}(B)}\int_{B} (f(y)-\essinf_{z\in B} f(z))d\omega_\mathfrak{K}(y)<\infty,
$$
where the supremum is taken over all the Euclidean balls $B$ in $\mathbb R^d$.

We now establish the main results of this paper.

Let $m\in\mathbb N$. Denote by $T_{t,m}^\mathfrak K$ the operator $T_{t,m}^\mathfrak K=t^m\partial _t^mT_t^\mathfrak K$, $t>0$. We consider the maximal operator $T_{*,m}^\mathfrak K$ defined by
$$
T_{*,m}^\mathfrak K(f)=\sup_{t>0}|T_{t,m}^\mathfrak K(f)|.
$$
Since $\{T_t^\mathfrak K\}_{t>0}$ is a diffusion semigroup in $L^p(\mathbb R^d,\omega_\mathfrak{K})$, $1\leq p<\infty$, according to \cite[Corollary 4.2]{LeMX2}, the operator $T_{*,m}^\mathfrak K$ is bounded from $L^p(\mathbb R^d,\omega_\mathfrak{K})$ into itself, for every $1<p<\infty$.

\begin{thm}\label{Th1.1}
Let $m\in\mathbb N$. The maximal operator $T_{*,m}^\mathfrak K$ is bounded from $L^1(\mathbb R^d,\omega_\mathfrak{K})$ into $L^{1,\infty}(\mathbb R^d,\omega_\mathfrak{K})$ and from $H^1(\Delta _\mathfrak K)$ into $L^1(\mathbb R^d,\omega_\mathfrak{K})$. Furthermore, if $f\in {\rm BMO}^\rho(\mathbb R^d,\omega_\mathfrak{K})$ and $T_{*,m}^\mathfrak K(f)(x)<\infty$, for almost all $x\in\mathbb R^d$, then $T_{*,m}^\mathfrak K(f)\in {\rm BLO}(\mathbb R^d,\omega_\mathfrak{K})$ and $\|T_{*,m}^\mathfrak K(f)\|_{{\rm BLO}(\mathbb R^d,\omega_\mathfrak{K})}\leq C\|f\|_{{\rm BMO}^\rho(\mathbb R^d,\omega_\mathfrak{K})}$ where $C>0$ does not depend on $f$.
\end{thm}

Suppose now that $m\in\mathbb N$, $m\geq 1$. We define the $m$-order Littlewood-Paley $g_m$ function by 
$$
g_m(f)(x)=\left(\int_0^\infty |T_{t,m}^\mathfrak K (f)(x)|^2\frac{dt}{t}\right)^{1/2},\quad x\in\mathbb R^d.
$$
Since $\{T_t^\mathfrak K \}_{t>0}$ is a symmetric diffusion semigroup from \cite[Corollary 1, p. 120]{St} it follows that $g_m$ is bounded from $L^p(\mathbb R^d,\omega_\mathfrak{K})$ into itself, for every $1<p<\infty$. In \cite[Theorem 1.2]{Li} it was proved that $g_1$ is bounded from $L^1(\mathbb R^d,\omega_\mathfrak{K})$ into $L^{1,\infty}(\mathbb R^d,\omega_\mathfrak{K})$. Dziuba\'nski and Hejna ($\!\!$\cite{DH5}) proved $L^p$ boundedness properties, $1<p<\infty$, for Littlewood-Paley functions defined by replacing the heat semigroup $\{T_t^\mathfrak K \}_{t>0}$ by families of Dunkl convolutions.

\begin{thm}\label{Th1.2}
Let $m\in\mathbb N$, $m\geq 1$. The Littlewood-Paley $g_m$ function is bounded from $L^1(\mathbb R^d,\omega_\mathfrak{K})$ into $L^{1,\infty}(\mathbb R^d,\omega_\mathfrak{K})$ and from $H^1(\Delta _\mathfrak K)$ into $L^1(\mathbb R^d,\omega_\mathfrak{K})$. Furthermore, if $f\in {\rm BMO}^\rho(\mathbb R^d,\omega_\mathfrak{K})$ and $g_m(f)(x)<\infty$, for almost all $x\in\mathbb R^d$, then $g_m(f)\in {\rm BLO}(\mathbb R^d,\omega_\mathfrak{K})$ and $\|g_m(f)\|_{{\rm BLO}(\mathbb R^d,\omega_\mathfrak{K})}\leq C\|f\|_{{\rm BMO}^\rho(\mathbb R^d,\omega_\mathfrak{K})}$ being $C>0$ independent of $f$.
\end{thm}

Suppose that $J\subset \mathbb R$. We consider a complex function $t\rightarrow a_t$ defined on $J$. Let $\sigma>0$. The $\sigma$-variation $V_\sigma(\{a_t\}_{t\in J})$ is defined by
$$
V_\sigma(\{a_t\}_{t\in J})=\sup_{\substack{t_n<...<t_1\\t_j\in J}}\Big(\sum_{j=1}^{n-1}|a_{t_j}-a_{t_{j+1}}|^\sigma\Big)^{1/\sigma},
$$
where the supremum is taken over all finite decreasing sequences $t_n<...<t_2<t_1$ with $t_j\in J$, $j=1,...,n$.

We consider a family of Lebesgue measurable functions $\{F_t\}_{t\in J}$ defined in $\mathbb R^d$. The $\sigma$-variation $V_\sigma(\{F_t\}_{t\in J})$ of $\{F_t\}_{t\in J}$ is the function defined by
$$
V_\sigma(\{F_t\}_{t\in J})(x):=V_\sigma(\{F_t(x)\}_{t\in J}),\quad x\in\mathbb R^d.
$$
If $J$ is a numerable subset of $\mathbb R$ then $V_\sigma(\{F_t\}_{t\in J})$ defines a Lebesgue measurable function. The measurability of $V_\sigma(\{F_t\}_{t\in J})$  can be also assumed when the function $t\rightarrow F_t(x)$ is continuous by considering the usual topology of $\mathbb R$, for almost all $x\in\mathbb R^d$, although $J$ does not satisfy the countable property.

L\'{e}pingle's inequality ($\!\!$\cite{Le}) concerning to bounded martingale sequence is a very useful tool in proving variational inequalities. Pisier and Xu ($\!\!$\cite{PX}) and Bourgain ($\!\!$\cite{Bou1}) gave simple proofs for L\'{e}pingle's inequality. The Bourgain's work ($\!\!$\cite{Bou1}) has motivated many authors to study variational inequalities for families of averages, semigroups of operators and truncating classical singular integral operators ($\!\!$\cite{AJS}, \cite{BORSS}, \cite{CJRW1}, \cite{CJRW2}, \cite{HMMT}, \cite{JKRW}, \cite{JSW}, \cite{JW}, \cite{LeMXu1} and \cite{MTX}).

In general, in order to obtain boundedness results for variational operators we need to consider $\sigma >2$. This is the case, for instance, with martingales ($\!\!$\cite{Q}) or with differentiation operators ($\!\!$\cite[Remark 1.7]{CJRW1}). The $\sigma$-variation operator $V_\sigma(\{F_t\}_{t>0})$ is related with the convergence of the family $\{F_t(x)\}_{t>0}$, $x\in\mathbb R^d$. In particular, if $x\in\mathbb R^d$ and $V_\sigma(\{F_t\}_{t>0})(x)<\infty$ there exists $\lim_{t\rightarrow t_0}F_t(x)$, for every $t_0\in (0,\infty)$. Suppose that $\{S_t\}_{t>0}$ is a family of bounded operators on $L^p(\mathbb R^d,\omega_\mathfrak{K})$, $1\leq p<\infty$. We define the $\sigma$-variation operator $V_\sigma(\{S_t\}_{t>0})$ of $\{S_t\}_{t>0}$ by
$$
V_\sigma (\{S_t\}_{t>0})(f)=V_\sigma (\{S_tf\}_{t>0}),\quad f\in L^p(\mathbb R^d,\omega_\mathfrak K).
$$
If $V_\sigma(\{S_t\}_{t>0}$ is bounded on $L^p(\mathbb R^d,\omega_\mathfrak{K})$ then there exists $\lim_{t\rightarrow t_0}S_tf(x)$ for almost all $x\in\mathbb R^d$ and all $t_0\in (0,\infty)$. The use of $\sigma$-variation operator instead of maximal operators in getting pointwise convergence has advantages, because the maximal operators require of a dense subset $\mathcal D\subset L^p(\mathbb R^d,\omega_\mathfrak{K})$ where the pointwise convergence holds. However, the use of the $\sigma$- variation operators for this purpose is more involved. 

When $\sigma=2$ a good substitute of the $\sigma$-variation operator is the oscillation operator defined as follows: Let $\{t_j\}_{j\in \mathbb N}$ be a decreasing sequence of positive numbers. If $\{S_t\}_{t>0}$ is as above we define the oscillation operator $\mathcal O(\{t_j\}_{j\in \mathbb N}, \{S_t\}_{t>0})$ as
$$
\mathcal O(\{t_j\}_{j\in \mathbb N}, \{S_t\}_{t>0})(f)(x)=\left(\sum_{j\in \mathbb N} \sup_{t_{j+1}\leq\varepsilon_{j+1}<\varepsilon_j\leq t_j}|S_{\varepsilon_j}(f)(x)-S_{\varepsilon_{j+1}}(f)(x)|^2\right)^{1/2}.
$$
According to \cite[Corollary 4.5]{LeMXu1} the $\sigma$-variation operator $V_\sigma(\{T_{t,m}^\mathfrak K \}_{t>0})$ is bounded from $L^p(\mathbb R^d,\omega_\mathfrak{K})$ into itself, for every $1<p<\infty$ and $m\in\mathbb N$. Furthermore, by \cite[p. 2091]{LeMXu1} the oscillation operator $\mathcal O(\{t_j\}_{j\in \mathbb N},\{T_{t,m}^\mathfrak K \}_{t>0})$ is bounded from $L^p(\mathbb R^d,\omega_\mathfrak{K})$ into itself, for every $1<p<\infty$ and every decreasing sequence $\{t_j\}_{j\in \mathbb N}$ of positive numbers.

\begin{thm}\label{Th1.3}
Let $m\in\mathbb N$ and $\sigma>2$. The $\sigma$-variation operator $V_\sigma(\{T_{t,m}^\mathfrak K \}_{t>0})$ is bounded from $L^1(\mathbb R^d,\omega_\mathfrak{K})$ into $L^{1,\infty}(\mathbb R^d,\omega_\mathfrak{K})$ and from $H^1(\Delta_\mathfrak K)$ into $L^1(\mathbb R^d,\omega_\mathfrak{K})$. Moreover, a function $f\in L^1(\mathbb R^d,\omega_\mathfrak{K})$ is in $H^1(\Delta_\mathfrak K)$ if and only if $V_\sigma(\{T_t^\mathfrak K \}_{t>0})(f)\in L^1(\mathbb R^d,\omega_\mathfrak{K})$ and the quantities  $\|f\|_{L^1(\mathbb R^d,\omega_\mathfrak{K})}+\|T_{*,0}^\mathfrak K (f)\|_{L^1(\mathbb R^d,\omega_\mathfrak{K})}$ and $\|f\|_{L^1(\mathbb R^d,\omega_\mathfrak{K})}+\|V_\sigma(\{T_t^\mathfrak K \}_{t>0})(f)\|_{L^1(\mathbb R^d,\omega_\mathfrak{K})}$ are equivalent.

If $f\in {\rm BMO}^\rho(\mathbb R^d,\omega_\mathfrak{K})$ and $V_\sigma(\{T_{t,m}^\mathfrak K \}_{t>0})(f)(x)<\infty$,  almost every $x\in\mathbb R^d$, then $V_\sigma(\{T_{t,m}^\mathfrak K \}_{t>0})(f)$ belongs to ${\rm BLO}(\mathbb R^d,\omega_\mathfrak{K})$ and 
$$\|V_\sigma(\{T_{t,m}^\mathfrak K \}_{t>0})(f)\|_{{\rm BLO}(\mathbb R^d,\omega_\mathfrak{K})}\leq C\|f\|_{{\rm BMO}^\rho (\mathbb R^d,\omega_\mathfrak{K})},$$
where $C>0$ does not depend on $f$.
\end{thm}

\begin{thm}\label{Th1.4}
Let $m\in\mathbb N$. Suppose that $\{t_j\}_{j\in \mathbb N}$ is a decreasing sequence of positive numbers. The oscillation operator operator $\mathcal O(\{t_j\}_{j\in \mathbb N},\{T_{t,m}^\mathfrak K \}_{t>0})$ is bounded from $L^1(\mathbb R^d,\omega_\mathfrak{K})$ into $L^{1,\infty}(\mathbb R^d,\omega_\mathfrak{K})$ and from $H^1(\Delta _\mathfrak K)$ into $L^1(\mathbb R^d,\omega_\mathfrak{K})$. 

Furthermore, if $f\in {\rm BMO}^\rho(\mathbb R^d,\omega_\mathfrak{K})$ and $\mathcal O(\{t_j\}_{j\in \mathbb N},\{T_{t,m}^\mathfrak K \}_{t>0})(f)(x)<\infty$,   for almost all $x\in\mathbb R^d$, then $\mathcal O(\{t_j\}_{j\in \mathbb N},\{T_{t,m}^\mathfrak K \}_{t>0})(f)\in {\rm BLO}(\mathbb R^d,\omega_\mathfrak{K})$ and 
$$\|\mathcal O(\{t_j\}_{j\in \mathbb N},\{T_{t,m}^\mathfrak K \}_{t>0})(f)\|_{{\rm BLO}(\mathbb R^d,\omega_\mathfrak{K})}\leq C\|f\|_{{\rm BMO}^\rho(\mathbb R^d,\omega_\mathfrak{K})},$$
where $C>0$ is independent of $f$.
\end{thm}

In the next sections we will prove our Theorems. Throughout this paper by $c$, $C$ we always denote positive constants that can change in each occurrence. 

\section{Proof of Theorems \ref{Th1.3} and \ref{Th1.4}}

We are going to prove Theorem \ref{Th1.3}. In order to establish the boundedness properties for the oscillation operator we can proceed analogously. 

\subsection{Variation operators on $L^1(\mathbb R^d,\omega_\mathfrak{K})$}
According to \cite[Corollary 6.1]{LeMXu1} (see also \cite[Theorem 3.3]{JoRe}) since $\{T_t^\mathfrak K \}_{t>0}$ is a diffusion semigroup of operators $V_\sigma(\{T_{t,m}^\mathfrak K\}_{t>0})$ is bounded from $L^p(\mathbb R ^d,\omega_\mathfrak{K})$ into itself, for every $1<p<\infty$.

By using again \eqref{1.5} we can see that
$$
\int_{\mathbb R^d}|t^m\partial _t^mT_t^\mathfrak K (x,y)||h(y)|d\omega_\mathfrak{K}(y)\leq \frac{C}{\omega _\mathfrak K(B(x,\sqrt{t}))}\|h\|_{L^1(\mathbb R^d,\omega _\mathfrak K)},\quad x\in \mathbb R^d\mbox{ and }t>0.
$$
Then, we can define $V_\sigma (\{T_{t,m}^\mathfrak K\}_{t>0})(h)(x)$, for every $x\in \mathbb R^d$. Moreover, by using dominated convergence theorem we can see that, for every $x\in \mathbb R^d$, the function $t\longrightarrow T_{t,m}(h)(x)$ is continuous in $(0,\infty )$. It follows that 
$$
V_\sigma (\{T_{t,m}\}_{t>0})(h)(x)=\sup_{\substack{0<t_n<...<t_1\\t_j\in \mathbb Q,j=1,...,n}}\Big(\sum_{j=1}^{n-1}|T_{t_j,m}^\mathfrak K (h)(x)-T_{t_{j+1},m}^\mathfrak K(h)(x)|^\sigma\Big)^{1/\sigma},\quad x\in \mathbb R^d.
$$
Since the set of finite subsets of $\mathbb Q$ is countable we conclude that the function $V_\sigma (\{T_{t,m}^\mathfrak K\}_{t>0})(h)$ is measurable on $\mathbb R^d$.

Let us see that the variation operator $V_\sigma(\{T_{t,m}^\mathfrak K\}_{t>0})$ is bounded from $L^1(\mathbb R ^d,\omega_\mathfrak{K})$ into $L^{1,\infty}(\mathbb R ^d,\omega_\mathfrak{K})$.

The triple $(\mathbb R^d,|\cdot|,\omega_\mathfrak{K})$ is a space of homogeneous type. According to the Calder\'on-Zygmund decomposition (see, for instance, \cite[Theorem 3.1]{BK} or \cite[Theorem 2.10]{Ai}), there exists $M>0$ such that, for every $f\in L^1(\mathbb R^d,\omega_\mathfrak{K})$ and $\lambda >0$, we can find a measurable function $g$, a sequence $(b_i)_{i\in \mathbb N}$ of measurable functions and a sequence $(B_i=B(x_i,r_i))_{i\in \mathbb N} $ of Euclidean balls satisfying that
\begin{itemize}
\item[(i)] $f=g+\sum_{i\in \mathbb N}b_i$;
\item[(ii)] $\|g\|_\infty \leq C\lambda$;
\item[(ii)] $\mbox{ supp } (b_i)\subset B_i^*$, $i\in \mathbb N$, and $\#\{j\in \mathbb N:x\in B_j^*\}\leq M$, for every $x\in \mathbb R^d$;
\item[(iv)] $\|b_i\|_{L^1(\mathbb R^d, \omega_\mathfrak K)}\leq C\lambda \omega_\mathfrak{K}(B_i)$, $i\in \mathbb N$;
\item[(v)] $\sum_{i\in \mathbb N}\omega_\mathfrak{K}(B_i)\leq \frac{C}{\lambda}\|f\|_{L^1(\mathbb R^d, \omega_\mathfrak K)}$.
\end{itemize}
Here $B^*$ represents the ball $B^*=B(x_0,\alpha r_0)$, when $B=B(x_0,r_0)$. The constants $C,\alpha$ do not depend on $f$.

Let $\lambda >0$ and $f\in L^1(\mathbb R^d,\omega_\mathfrak{K})\cap L^2(\mathbb R^d,\omega_\mathfrak{K})$. We write $f=g+b$, where $b=\sum_{i\in \mathbb N}b_i$ as above. The series defining $b(x)$ is actually a finite sum for every $x\in \mathbb R^d$. Furthermore, the series $\sum_{i\in \mathbb N} b_i$ converges in $L^1(\mathbb R^d,\omega_\mathfrak{K})$ and 
$$
\|b\|_{L^1(\mathbb R^d, \omega_\mathfrak K)}\leq \sum_{i\in \mathbb N}\|b_i\|_{L^1(\mathbb R^d, \omega_\mathfrak K)} \leq C\lambda \sum_{i\in \mathbb N}\omega_\mathfrak{K}(B_i)\leq C\|f\|_{L^1(\mathbb R^d, \omega_\mathfrak K)}.
$$
Then, $g\in L^\infty (\mathbb R^d,\omega_\mathfrak{K})\cap L^1(\mathbb R^d,\omega_\mathfrak{K})$, and hence, $g\in L^q(\mathbb R^d,\omega_\mathfrak{K})$, for every $1\leq q\leq \infty$. It follows that $b\in L^1(\mathbb R^d,\omega_\mathfrak{K})\cap L^2(\mathbb R^d,\omega_\mathfrak{K})$.

We get
\begin{align*}
\omega_\mathfrak{K}(\{x\in\mathbb R^d: V_\sigma(\{T_{t,m}^\mathfrak K\}_{t>0})(f)(x)>\lambda \})&\leq \omega_\mathfrak{K}\big(\{x\in \mathbb R^d:V_\sigma(\{T_{t,m}^\mathfrak K\}_{t>0})(g)(x)>\frac{\lambda }{2}\}\big)\\
&\quad +\omega_\mathfrak{K}\big(\{x\in \mathbb R^d: V_\sigma(\{T_{t,m}^\mathfrak K\}_{t>0})(b)(x)>\frac{\lambda }{2}\}\big).
\end{align*}

By considering (ii) and that $V_\sigma(\{T_{t,m}^\mathfrak K\}_{t>0})$ is bounded on $L^2(\mathbb R^d,\omega_\mathfrak K )$ we obtain
\begin{align*}
    \omega_\mathfrak{K}\big(\{x\in \mathbb R^d: V_\sigma(\{T_{t,m}^\mathfrak K\}_{t>0})(g)(x)>\frac{\lambda }{2}\}\big)&\leq \frac{4}{\lambda ^2}\int_{\mathbb R^d}|V_\sigma(\{T_{t,m}^\mathfrak K\}_{t>0})(g)(x)|^2d\omega_\mathfrak{K}(x)\\
    &\hspace{-2cm}\leq \frac{C}{\lambda ^2}\|g\|_{L^2(\mathbb R^d, \omega_\mathfrak K)}^2\leq \frac{C}{\lambda }\|g\|_{L^1(\mathbb R^d,\omega_\mathfrak K)}\leq \frac{C}{\lambda }\|f\|_{L^1(\mathbb R^d,\omega_\mathfrak K)}.
\end{align*}
Our aim is then to establish that
\begin{equation}\label{aim}
     \omega_\mathfrak{K}\big(\{x\in \mathbb R^d: V_\sigma(\{T_{t,m}^\mathfrak K\}_{t>0})(b)(x)>\frac{\lambda }{2}\}\big)\leq \frac{C}{\lambda }\|f\|_{L^1(\mathbb R^d,\omega_\mathfrak K)}.
\end{equation}
Consider $s_i=r_i^2$, $i\in \mathbb N$. Since, for every $s>0$, $T_s^\mathfrak K$ is contractive in $L^1(\mathbb R^d,\omega_\mathfrak{K})$ it follows that
$$
\sum_{i\in \mathbb N}\|T_{s_i}^\mathfrak K (b_i)\|_{L^1(\mathbb R^d, \omega_\mathfrak K)}\leq \sum_{i\in \mathbb N}\|b_i\|_{L^1(\mathbb R^d, \omega_\mathfrak K)}\leq C\|f\|_{L^1(\mathbb R^d, \omega_\mathfrak K)},
$$
and the series $\sum_{i\in \mathbb N}|T_{s_i}^\mathfrak K  (b_i)|$ and $\sum_{i\in \mathbb N}|(I-T_{s_i}^\mathfrak K )b_i|$ converge in $L^1(\mathbb R^d,\omega_\mathfrak{K})$. Thus we can write
$$
b=\sum_{i\in \mathbb N}T_{s_i}^\mathfrak K   (b_i)+\sum_{i\in \mathbb N}(I-T_{s_i}^\mathfrak K   )b_i,
$$
and
$$
T_{t,m}^\mathfrak K (b)=T_{t,m}^\mathfrak K \Big(\sum_{i\in \mathbb N} T_{s_i}^\mathfrak K   (b_i)\Big)+T_{t,m}^\mathfrak K \Big(\sum_{i\in \mathbb N}(I-T_{s_i}^\mathfrak K   )b_i\Big),\quad t>0.
$$
Then,
$$
V_\sigma(\{T_{t,m}^\mathfrak K\}_{t>0})(b)\leq V_\sigma(\{T_{t,m}^\mathfrak K\}_{t>0})\Big(\sum_{i\in \mathbb N}T_{s_i}^\mathfrak K   (b_i)\Big)+V_\sigma (\{T_{t,m}^\mathfrak K\}_{t>0})\Big(\sum_{i\in \mathbb N}(I-T_{s_i}^\mathfrak K   )b_i\Big),
$$
and
\begin{align}\label{aim1}
    \omega_\mathfrak{K}\Big(\{x\in \mathbb R^d: V_\sigma(\{T_{t,m}^\mathfrak K\}_{t>0})(b)(x)>\frac{\lambda }{2}\}\Big)&\nonumber\\
    &\hspace{-4cm}\leq \omega_\mathfrak{K}\Big(\{x\in \mathbb R^d: V_\sigma(\{T_{t,m}^\mathfrak K\}_{t>0})\Big(\sum_{i\in \mathbb N}T_{s_i}^\mathfrak K(b_i)\Big)(x)>\frac{\lambda }{4}\}\Big)\nonumber\\
    &\hspace{-4cm}\quad +\omega_\mathfrak{K}\Big(\{x\in \mathbb R^d: V_\sigma(\{T_{t,m}^\mathfrak K\}_{t>0})\Big(\sum_{i\in \mathbb N} (I-T_{s_i}^\mathfrak K)b_i\Big)(x)>\frac{\lambda }{4}\}\Big).
\end{align}

Suppose that $m,n\in \mathbb{N}$, $n<m$. By proceeding as in \cite[pp. 11-12]{Li} we obtain that
$$
\Big\|\sum_{i=n}^mT_{s_i}^\mathfrak K (b_i)\Big\|_{L^2(\mathbb R^d, \omega_\mathfrak K)} \leq \Big\|\sum_{i=n}^mT_{s_i}^\mathfrak K (|b_i|)\Big\|_{L^2(\mathbb R^d, \omega_\mathfrak K)}\leq C\lambda \big(\sum_{i=n}^m\omega_\mathfrak{K}(B_i)\big)^{1/2}.
$$
It follows that the series $\sum_{i\in \mathbb N} T_{s_i}^\mathfrak K (b_i)$ converges in $L^2(\mathbb R ^d,\omega_\mathfrak{K})$ and
$$
\Big\|\sum_{i\in \mathbb N} T_{s_i}^\mathfrak K (b_i)\Big\|_{L^2(\mathbb R^d, \omega_\mathfrak K)} \leq C\lambda \Big(\sum_{i\in \mathbb N} \omega_\mathfrak{K}(B_i)\Big)^{1/2}\leq C\sqrt{\lambda \|f\|_{L^1(\mathbb R^d, \omega_\mathfrak K)} }. 
$$
Since $V_\sigma(\{T_{t,m}^\mathfrak K\}_{t>0})$ is bounded on $L^2(\mathbb R^d,\omega_\mathfrak{K})$ we deduce that
\begin{equation}\label{aim2}
\omega_\mathfrak{K}\Big(\{x\in \mathbb R^d: V_\sigma(\{T_{t,m}^\mathfrak K\}_{t>0})\Big(\sum_{i\in \mathbb N}T_{s_i}^\mathfrak K (b_i)\Big)(x)>\frac{\lambda }{4}\}\Big)\leq \frac{C}{\lambda}\|f\|_{L^1(\mathbb R^d, \omega_\mathfrak K)} .
\end{equation}

On the other hand, 
$$
V_\sigma(\{T_{t,m}^\mathfrak K\}_{t>0})\Big(\sum_{i\in \mathbb N} (I-T_{s_i}^\mathfrak K)b_i\Big)(x)\leq \sum_{i\in \mathbb N} V_\sigma(\{T_{t,m}^\mathfrak K\}_{t>0})((I-T_{s_i}^\mathfrak K )b_i)(x),\quad  a.e. x\in \mathbb R^d.
$$

Indeed, let $0<t_n<...<t_2<t_1$, $t_j\in \mathbb Q$, $j=1,...,n$, $n\in \mathbb N$. Since $T_{t,m}^\mathfrak K$, $t>0$, is bounded on $L^1(\mathbb R^d,\omega_\mathfrak{K})$ we obtain
$$
T_{t,m}^\mathfrak K  \Big(\sum_{i\in \mathbb N}(I-T_{s_i}^\mathfrak K   )b_i\Big)_{|t=t_j}=\sum_{i\in \mathbb N} T_{t,m}^\mathfrak K  ((I-T_{s_i}^\mathfrak K   )b_i)_{|t=t_j},\quad j=1,...,n,
$$
in $L^1(\mathbb R^d,\omega_\mathfrak{K})$.

We get, for almost every $x\in \mathbb R^d$,
\begin{align*}
\Big(\sum_{j=1}^{n-1}\Big|T_{t,m}^\mathfrak K  \big(\sum_{i\in \mathbb N}(I-T_{s_i}^\mathfrak K   )b_i\big)(x)_{|t=t_j}-T_{t,m}^\mathfrak K  \big(\sum_{i\in \mathbb N}(I-T_{s_i}^\mathfrak K   )b_i\big)(x)_{|t=t_{j+1}}\Big|^\sigma \Big)^{1/\sigma }&\\
&\hspace{-10cm}=\Big(\sum_{j=1}^{n-1}\Big|\big(\sum_{i\in \mathbb N}T_{t,m}^\mathfrak K  (I-T_{s_i}^\mathfrak K )b_i\big)(x)_{|t=t_j}-\big(\sum_{i\in \mathbb N}T_{t,m}^\mathfrak K  (I-T_{s_i}^\mathfrak K )b_i\big)(x)_{|t=t_{j+1}}\Big|^\sigma \Big)^{1/\sigma }\\
&\hspace{-10cm}=\Big(\sum_{j=1}^{n-1}\Big|\Big[\sum_{i\in \mathbb N}\big(T_{t,m}^\mathfrak K  ((I-T_{s_i}^\mathfrak K )b_i)_{|t=t_j}-T_{t,m}^\mathfrak K  ((I-T_{s_i}^\mathfrak K )b_i)_{|t=t_{j+1}}\big)\Big](x)\Big|^\sigma \Big)^{1/\sigma }\\
&\hspace{-10cm}\leq \Big(\sum_{j=1}^{n-1}\Big(\Big[\sum_{i\in \mathbb N}\Big|T_{t,m}^\mathfrak K  ((I-T_{s_i}^\mathfrak K )b_i)_{|t=t_j}-T_{t,m}^\mathfrak K  ((I-T_{s_i}^\mathfrak K )b_i)_{|t=t_{j+1}}\Big|\Big](x)\Big)^\sigma \Big)^{1/\sigma }\\
&\hspace{-10cm}\leq \sum_{i\in \mathbb N}\Big(\sum_{j=1}^{n-1}\big|T_{t,m}^\mathfrak K  ((I-T_{t_i}^\mathfrak K )b_i)_{|t=t_j}(x)-T_{t,m}^\mathfrak K  ((I-T_{s_i}^\mathfrak K )b_i)_{|t=t_{j+1}}(x)\big|^\sigma \Big)^{1/\sigma }.
\end{align*}
It is remarkable that, for $j=1,...,n$, the series 
$$
\sum_{i\in \mathbb N}\big|T_{t,m}^\mathfrak K  ((I-T_{s_i}^\mathfrak K )b_i)_{|t=t_j}-T_{t,m}^\mathfrak K  ((I-T_{s_i}^\mathfrak K )b_i)_{|t=t_{j+1}}\big|,
$$
converges in $L^1(\mathbb R^d,\omega_\mathfrak{K})$, and then, for $j=1,...,n$,
\begin{align*}
    \Big[\sum_{i\in \mathbb N}\Big|T_{t,m} ^\mathfrak K  ((I-T_{s_i}^\mathfrak K )b_i)_{|t =t _j}-T_{t,m} ^\mathfrak K ((I-T_{s_i}^\mathfrak K )b_i)_{|t =t _{j+1}}\Big|\Big](x)\\
    &\hspace{-7cm}=\sum_{i\in \mathbb N} \big|T_{t,m} ^\mathfrak K  ((I-T_{s_i}^\mathfrak K )b_i)_{|t =t _j}(x)-T_{t,m} ^\mathfrak K  ((I-T_{s_i}^\mathfrak K )b_i)_{|t =t _{j+1}}(x)\big|,\quad \mbox{ a.e. }x\in \mathbb R^d.
\end{align*}
It can be concluded that
\begin{align*}
    V_\sigma(\{T_{t,m}^\mathfrak K\}_{t>0})\Big(\sum_{i\in \mathbb N} (I-T_{s_i}^\mathfrak K )b_i\Big)(x)\\
    &\hspace{-4.5cm}=\sup_{\substack{0<t _n<...<t_1\\t_j\in \mathbb Q,j=1,...,n}}\Big(\sum_{j=1}^{n-1}\big|T_{t,m}^\mathfrak K  \Big(\sum_{i\in \mathbb N} (I-T_{s_i}^\mathfrak K )b_i\Big)_{|t =t _j}(x)-T_{t,m}^\mathfrak K  \Big(\sum_{i\in \mathbb N} (I-T_{s_i}^\mathfrak K )b_i\Big)_{|t =t _{j+1}}(x)\big|^\sigma \Big)^{1/\sigma }\\
    &\hspace{-4.5cm} \leq \sum_{i\in \mathbb N} V_\sigma(\{T_{t,m}^\mathfrak K\}_{t>0})((I-T_{s_i}^\mathfrak K )b_i)(x),\quad \mbox{ a.e. }x\in \mathbb R^d,
\end{align*}
and, thus, it follows that
\begin{align*}
\omega_\mathfrak{K}\Big(\big\{x\in \mathbb R^d: V_\sigma(\{T_{t,m}^\mathfrak K\}_{t>0})\Big(\sum_{i\in \mathbb N} (I-T_{s_i}^\mathfrak K)b_i\Big)(x)>\frac{\lambda }{4}\big\}\Big)\\
&\hspace{-4cm}\leq \omega_\mathfrak{K}\Big(\big\{x\in \mathbb R^d: \sum_{i\in \mathbb N}V_\sigma(\{T_{t,m}^\mathfrak K\}_{t>0})\big((I-T_{s_i}^\mathfrak K )b_i\big)(x)>\frac{\lambda }{4}\big\}\Big).
\end{align*}

We are going to see that
\begin{equation}\label{aim3}
\omega_\mathfrak{K}\Big(\big\{x\in \mathbb R^d: \sum_{i\in \mathbb N} V_\sigma(\{T_{t,m}^\mathfrak K\}_{t>0})((I-T_{s_i}^\mathfrak K )b_i)(x)>\frac{\lambda }{4}\big\}\Big)\leq \frac{C}{\lambda}\|f\|_{L^1(\mathbb R^d, \omega_\mathfrak K)},
\end{equation}
which, jointly \eqref{aim1} and \eqref{aim2}, leads to \eqref{aim}.

We have that
\begin{align*}
    \omega_\mathfrak{K}\Big(\big\{x\in \mathbb R^d: \sum_{i\in \mathbb N} V_\sigma(\{T_{t,m}^\mathfrak K\}_{t>0})((I-T_{s_i}^\mathfrak K )b_i)(x)>\frac{\lambda }{4}\big\}\Big)&\leq \omega_\mathfrak{K}\big(\bigcup_{i\in \mathbb N} \theta (2B_i^*)\big)\\
    &\hspace{-7cm}+\omega_\mathfrak{K}\Big(\big\{x\in \bigcap_{i\in \mathbb N} (\theta (2B_i^*))^c:\sum_{i\in \mathbb N} V_\sigma(\{T_{t,m}^\mathfrak K\}_{t>0})((I-T_{s_i}^\mathfrak K )b_i)(x)>\frac{\lambda }{4}\big\}\Big).
\end{align*}
From \eqref{1.1} and \eqref{1.2} we get
$$
\omega_\mathfrak{K}\big(\bigcup_{i\in \mathbb N} \theta (2B_i^*)\big)\leq C\sum_{i\in \mathbb N} \omega_\mathfrak{K}(B_i)\leq \frac{C}{\lambda }\|f\|_{L^1(\mathbb R^d, \omega_\mathfrak K)} .
$$
On the other hand, observe that, for $F\in L^1(\mathbb R^d,\omega_\mathfrak K)$ defined on $\mathbb R^d$, we can write
\begin{align}\label{Vderivt}
    V_\sigma(\{T_{t,m}^\mathfrak K\}_{t>0})(F)(x)&\leq \int_0^\infty |\partial _t[T_{t,m}^\mathfrak KF(x)]|dt\nonumber\\
    &\leq \int_0^\infty (m|T_{t,m}^\mathfrak KF(x)|+|T_{t,m+1}^\mathfrak KF(x)|)\frac{dt}{t},\quad x\in \mathbb R^d.
\end{align}
Then,
\begin{align*}
   & \omega_\mathfrak{K}\Big(\big\{x\in \bigcap_{i\in \mathbb N} (\theta (2B_i^*))^c:\sum_{i\in \mathbb N} V_\sigma(\{T_{t,m}^\mathfrak K\}_{t>0})((I-T_{s_i}^\mathfrak K )b_i)(x)>\frac{\lambda }{4}\big\}\Big)\\
    &\leq \frac{4}{\lambda }\sum_{i\in \mathbb N} \int_{(\theta (2B_i^*))^c}V_\sigma(\{T_{t,m}^\mathfrak K\}_{t>0})((I-T_{s_i}^\mathfrak K )b_i)(x)d\omega_\mathfrak{K}(x)\\
    &\leq \frac{C}{\lambda }\sum_{i\in \mathbb N} \int_{(\theta (2B_i^*))^c}\int_0^\infty (|T_{t,m}^\mathfrak K  ((I-T_{s_i}^\mathfrak K )b_i)(x)|+|T_{t,m+1}^\mathfrak K  ((I-T_{s_i}^\mathfrak K )b_i)(x)|)\frac{dt}{t}d\omega_\mathfrak{K}(x)\\
    &\leq \frac{C}{\lambda }\sum_{i\in \mathbb N} \int_{B_i^*}|b_i(y)|(J_i^m(y)+J_i^{m+1}(y))d\omega_\mathfrak K(y),
\end{align*}
where, for every $i\in \mathbb N$, $J_i^0=0$ and for $r\in \mathbb N$, $r\geq 1$,
$$
J_i^r(y)=\int_{(\theta (2B_i^*))^c}\int_0^\infty t^{r-1}|\partial _t^r(T_t^\mathfrak K  (x,y)-T_{t+s_i}^\mathfrak K(x,y))|dtd\omega_\mathfrak{K}(x),\quad y\in B_i^*.
$$
We are going to see that, for $r\in \mathbb N$, $r\geq 1$, $\sup_{i\in \mathbb N}\sup_{y\in B_i^*}J_i^r(y)<\infty$.

Let $i\in \mathbb N$ and $r\in \mathbb N$, $r\geq 1$. We write $J_i^r=\sum_{\ell \in \mathbb N}\mathbb J_{i,\ell}^r$, where, for every $\ell \in \mathbb N$, 
$$
\mathbb{J}_{i ,\ell}^r(y)=\int_{(\theta (2B_i^*))^{\rm c}}\int_{\ell s_i}^{(\ell +1)s_i}t^{r-1}|\partial _t^r (T_t^\mathfrak K  (x,y)-T_{t+s_i}^\mathfrak K(x,y))|dtd\omega_\mathfrak{K}(x),\quad y\in B_i^*.
$$
Since $\partial _u^rT_u^\mathfrak K (x,y)=\Delta _{\mathfrak K,x}^rT_u^\mathfrak K(x,y)$, $x,y\in \mathbb R ^d$, $u>0$, we get
\begin{align*}
    \mathbb J_{i,\ell}^r(y)&=\int_{(\theta (2B_i^*))^{\rm c}}\int_{\ell s_i}^{(\ell+1)s_i}t^{r-1}\Big|\int_t^{t+s_i}\Delta_{\mathfrak K,x}^{r+1}T_u^\mathfrak K  (x,y)du\Big|dtd\omega_\mathfrak{K}(x)\\
    &\leq \int_{\ell s_i}^{(\ell +1)s_i}t^{r-1}\int_t^{t+s_i}\int_{(\theta (2B_i^*))^{\rm c}}|\Delta_{\mathfrak K,x}  ^{r+1}T_u^\mathfrak K (x,y)|d\omega_\mathfrak{K}(x)dudt,\quad y\in B_i^*.
\end{align*}
For every $y\in B_i^*$, $x\in \mathbb R^d$ and $g\in G$, we have that $|gx-x_i|\leq |gx-y|+|y-x_i|$. Then, $\rho(x,x_i)\leq \rho(x,y)+\alpha r_i$ and $(\theta (2B_i^*))^{\rm c}\subset (\theta (B(y,\alpha r_i)))^{\rm c}$, $x\in \mathbb R^d$, $y\in B_i^*$.

By using now \cite[Lemmas 2.3 and 2.6]{Li} we obtain, for small enough $\varepsilon>0$, $y\in B_i^*$ and $u\in (\ell s_i,(\ell+1)s_i)$,
\begin{align}\label{cLi}
    \int_{(\theta (2B_i^*))^{\rm c}}|\Delta_\mathfrak K  ^{r+1}T_u^\mathfrak K  (x,y)|d\omega_\mathfrak{K}(x)&\nonumber\\
  &\hspace{-4cm}\leq \left(\int_{(\theta (B(y,\alpha r_i)))^{\rm c}}|\Delta_\mathfrak K  ^{r+1}T_u^\mathfrak K  (x,y)|^2e^{\varepsilon \frac{\rho(x,y)^2}{u}}d\omega_\mathfrak{K}(x)\right)^{1/2}\left( \int_{(\theta (B(y,\alpha r_i)))^{\rm c}}e^{-\varepsilon \frac{\rho(x,y)^2}{u}}d\omega_\mathfrak{K}(x)\right)^{1/2}\nonumber\\
    &\hspace{-4cm}\leq C\Big(\frac{e^{-\alpha^2\varepsilon \frac{r_i^2}{u}}}{u^{2r+2}\omega_\mathfrak{K}(B(y,\sqrt{u}))}\Big)^{1/2}\Big(\omega_\mathfrak{K}(B(y,\sqrt{u}))e^{-\alpha ^2\varepsilon \frac{r_i^2}{2u}}\Big)^{1/2}=C\frac{e^{-c\frac{s_i}{u}}}{u^{r+1}}.
\end{align}
We get, if $\ell \geq 1$,
$$
    \mathbb J_{i,\ell}^r(y)\leq C\int_{\ell s_i}^{(\ell+1)s_i}t^{r-1}\int_t^{t+s_i}\frac{e^{-c\frac{s_i}{u}}}{u^{r+1}}dudt\leq Cs_i\int_{\ell s_i}^{(\ell+1)s_i}\frac{dt}{t^2}\leq \frac{C}{\ell^2},\quad y\in B_i^*.
$$
On the other hand, we can write
\begin{align*}
    \mathbb J_{i,0}^r(y)&\leq C\int_0^{s_i}t^{r-1}\int_t^{t+s_i}\Big(\frac{u}{s_i}\Big)^{3/2}\frac{du}{u^{r+1}}dt\leq \frac{C}{s_i^{3/2}}\int_0^{s_i}t^{r-1}\int_t^{t+s_i}\frac{du}{u^{r-1/2}}dt\\
    &\leq \frac{C}{\sqrt{s_i}}\int_0^{s_i}\frac{dt}{\sqrt{t}}=C,\quad y\in B_i^*.
\end{align*}

We conclude that $J_i^r(y)=\sum_{\ell\in \mathbb N} \mathbb J_{i,\ell}^r (y)\leq C$, $y\in B_i^*$, where $C>0$ does not depend on $i$.

We obtain 
\begin{align*}
\omega_\mathfrak{K}\Big(\big\{x\in \bigcap_{i\in \mathbb N} (\theta (2B_i^*))^{\rm c}:\sum_{i\in \mathbb N} V_\sigma (\{T_{t,\ell }^\mathfrak K\}_{t>0})((I-T_{s_i}^\mathfrak K )b_i)(x)>\frac{\lambda }{4}\big\}\Big)&\\
&\hspace{-6cm} \leq \frac{C}{\lambda }\sum_{i\in \mathbb N} \int_{B_i^*}|b_i(y)|d\omega_\mathfrak{K}(y)\leq \frac{C}{\lambda}\|f\|_{L^1(\mathbb R^d, \omega_\mathfrak K)},
\end{align*}
and \eqref{aim3} is thus established.

\subsection{Variation operators on $H^1(\Delta _\mathfrak K)$}
We prove first that there exists $C>0$ such that, for every $f\in H^1(\Delta _\mathfrak K)$,
$$
\|V_\sigma (\{T_{t,m}^\mathfrak K\}_{t>0})(f)\|_{L^1(\mathbb R^d, \omega_\mathfrak K)} \leq C\|f\|_{H^1(\Delta _\mathfrak K)} .
$$
It is sufficient to see that there exists $C>0$ such that, for every $(1,2,\Delta_\mathfrak K,1)$-atom $a$,
$$
\|V_\sigma (\{T_{t,m}^\mathfrak K\}_{t>0})(a)\|_{L^1(\mathbb R^d, \omega_\mathfrak K)} \leq C.
$$
Let $a$ be a $(1,2,\Delta_\mathfrak K ,1)$-atom. There exist $b\in D(\Delta_\mathfrak K )$ and a ball $B=B(x_B,r_B)$ satisfying that
\begin{itemize}
    \item[(i)] $a=\Delta_\mathfrak K  b$;
    \item[(ii)] ${\rm supp} (\Delta_\mathfrak K ^\ell  b)\subset \theta(B)$, $\ell =0,1$;
    \item[(iii)] $\|(r_B^2\Delta_\mathfrak K )^\ell b\|_{L^2(\mathbb R^d,\omega_\mathfrak K)}\leq r_B^2\omega_\mathfrak{K}(B)^{-1/2}$, $\ell =0,1$.
\end{itemize}
We decompose $\|V_\sigma (\{T_{t,m}^\mathfrak K\}_{t>0})(a)\|_{L^1(\mathbb R^d, \omega_\mathfrak K)} $ as follows
$$
\|V_\sigma (\{T_{t,m}^\mathfrak K\}_{t>0})(a)\|_{L^1(\mathbb R^d, \omega_\mathfrak K)} =\left(\int_{\theta(4B)}+\int_{(\theta(4B))^{\rm c}}\right)V_\sigma (\{T_{t,m}^\mathfrak K\}_{t>0})(a)(x)d\omega_\mathfrak{K}(x)=I_1(a)+I_2(a).
$$
Since $V_\sigma (\{T_{t,m}^\mathfrak K\}_{t>0})$ is bounded on $L^2(\mathbb R^d,\omega_\mathfrak{K})$, using \eqref{1.1}, \eqref{1.2} and (iii) for $\ell =1$ we obtain
$$
I_1(a)\leq \|V_\sigma (\{T_{t,m}^\mathfrak K\}_{t>0})(a)\|_{L^2(\mathbb R^d,\omega _\mathfrak K)}\omega_\mathfrak{K}(\theta(4B))^{1/2}\leq C\|a\|_{L^2(\mathbb R^d,\omega_\mathfrak{K})}\omega_\mathfrak K(B)^{1/2}\leq C.
$$

On the other hand, by taking into account \eqref{Vderivt} we have that
$V_\sigma (\{T_{t,m}^\mathfrak K\}_{t>0})(a)(x)\leq C(J_m(x)+J_{m+1}(x))$, $x\in \mathbb R^d$, where $J_0=0$ and for every $r\in \mathbb N$, $r\geq 1$,
$$
J_r(x)=\int_0^\infty |T_{t,r}^\mathfrak K (a)(x)|\frac{dt}{t},\quad x\in \mathbb R^d.
$$
Let $r\in \mathbb N$, $r\geq 1$. Since $a=\Delta_\mathfrak K  b$, it follows that $|T_{t,r}^\mathfrak K (a)|=|t^r\partial_t^{r+1}T_t^\mathfrak K (b)|=|t^r\Delta_\mathfrak K^{r+1}b|$. Then, as in \eqref{cLi}, according to \cite[Lemmas 2.3 and 2.6]{Li} we obtain
\begin{align*}
    \int_{(\theta (4B))^{\rm c}}J_r(x)d\omega_\mathfrak{K}(x)&\leq \int_0^\infty t^{r-1}\int_{(\theta(4B))^{\rm c}}|\Delta_\mathfrak K^{r+1}T_t^\mathfrak K (b)(x)|d\omega_\mathfrak{K}(x)dt\\
    &\leq C\int_{\theta (B)}|b(y)|\int_0^\infty \frac{e^{-c\frac{r_B^2}{t}}}{t^2}dtd\omega_\mathfrak{K}(y)\leq \frac{C}{r_B^2}\|b\|_{L^2(\mathbb R^d,\omega_\mathfrak K)}\omega_\mathfrak{K}(B)^{1/2}\leq C.
\end{align*}
It follows that $I_2(a)\leq C$ and it can be concluded that 
$
\|V_\sigma (\{T_{t,m}^\mathfrak K\}_{t>0})(a)\|_{L^1(\mathbb R^d, \omega_\mathfrak K)} \leq C,
$
where $C>0$ does not depend on $a$.

Let us see now that $\|f\|_{L^1(\mathbb R^d, \omega_\mathfrak K)} +\|T_{*,0}^\mathfrak K  (f)\|_{L^1(\mathbb R^d, \omega_\mathfrak K)}$ and $\|f\|_{L^1(\mathbb R^d, \omega_\mathfrak K)} +\|V_\sigma (\{T_t^\mathfrak K\}_{t>0})(f)\|_{L^1(\mathbb R^d, \omega_\mathfrak K)}$ are equivalent. It is clear that
$$
T_{*,0}^\mathfrak K (f) \leq V_\sigma (\{T_t^\mathfrak K \}_{t>0})(f)+|T_1^\mathfrak K (f)|.
$$
By using \cite[Theorem 2.2]{ADH} and that $T_1^\mathfrak K  $ is bounded from $L^1(\mathbb R^d,\omega_\mathfrak{K})$ into itself, we deduce that if $f\in L^1(\mathbb R^d,\omega_\mathfrak{K})$ and $V_\sigma (\{T_t^\mathfrak K \}_{t>0})(f)\in L^1(\mathbb R^d,\omega_\mathfrak{K})$, then $f\in H^1(\Delta_\mathfrak K  )$. Moreover, as it has just been proved, $V_\sigma (\{T_t^\mathfrak K \}_{t>0})$ is bounded from $H^1(\Delta_\mathfrak K )$ into $L^1(\mathbb R^d,\omega_\mathfrak{K})$, then we can conclude that if $f\in L^1(\mathbb R^d,\omega_\mathfrak{K})$, then $f\in H^1(\Delta_\mathfrak K)$ if and only if $V_\sigma (\{T_t\}_{t>0})(f)\in L^1(\mathbb R^d,\omega_\mathfrak{K})$. 

We also have that
$$
T_{*,0}^\mathfrak K (f)\leq V_\sigma (\{T_t^\mathfrak K \}_{t>0})(f)+|T_\varepsilon ^\mathfrak K  f|,\quad \varepsilon >0.
$$
By taking into account that, for every $f\in L^1(\mathbb R^d,\omega_\mathfrak{K})$, $\lim_{\varepsilon \rightarrow 0^+}T_\varepsilon ^\mathfrak K  (f)(x)=f(x)$, for almost all $x\in \mathbb R^d$, there exists $C>0$ for which
\begin{align*}
\frac{1}{C}\big(\|f\|_{L^1(\mathbb R^d, \omega_\mathfrak K)} +\|T_{*,0}^\mathfrak K  (f)\|_{L^1(\mathbb R^d, \omega_\mathfrak K)}\big) &\leq \|f\|_{L^1(\mathbb R^d, \omega_\mathfrak K)} +\|V_\sigma (\{T_t^\mathfrak K\}_{t>0})(f)\|_{L^1(\mathbb R^d, \omega_\mathfrak K)}\\
&\leq C\big(\|f\|_{L^1(\mathbb R^d, \omega_\mathfrak K)} +\|T_{*,0}^\mathfrak K  (f)\|_{L^1(\mathbb R^d, \omega_\mathfrak K)} \big),
\end{align*}
for every $f\in H^1(\Delta_\mathfrak K)$.

\subsection{Variation operators on ${\rm BMO}^\rho (\mathbb R^d,\omega_\mathfrak{K})$}\label{S2.3}

Suppose that $f\in {\rm BMO}^\rho (\mathbb R^d,\omega_\mathfrak{K})$ satisfies that $V_\sigma (\{T_{t,m}^\mathfrak K\}_{t>0})(f)(x)<\infty$, for almost all $x\in \mathbb R^d$.

Our first objective is to establish that $V_\sigma (\{T_{t,m}^\mathfrak K\}_{t>0})(f)\in {\rm BMO}(\mathbb R^d,\omega_\mathfrak{K})$.

Let $x_0\in \mathbb R^d$. The function $\phi (t)=T_t^\mathfrak K (f)(x_0)$, $t\in (0,\infty )$, is smooth and, for every $\ell\in \mathbb N$,
$$
\phi ^{(\ell)}(t)=\int_{\mathbb R^d}\partial _t^\ell T_t^\mathfrak K (x_0,y)f(y)d\omega_\mathfrak{K}(y),\quad t\in (0,\infty ).
$$
Indeed, let $t>0$ and $B_t=B(x_0,\sqrt{t})$. According to \eqref{1.1}, \eqref{1.2}, \eqref{1.4} and \eqref{1.5} we obtain
\begin{align*}
    \int_{\mathbb R^d}|T_t^\mathfrak K (x_0,y)f(y)|d\omega_\mathfrak{K}(y)&\leq |f_{\theta (B_t)}|+\int_{\mathbb R^d}T_t^\mathfrak K (x_0,y)|f(y)-f_{\theta (B_t)}|d\omega_\mathfrak{K}(y)\\
    &\leq |f_{\theta (B_t)}|+\frac{C}{\omega_\mathfrak{K}(\theta (B_t))}\int_{\mathbb R^d}e^{-c\frac{\rho(x_0,y)^2}{t}}|f(y)-f_{\theta (B_t)}|d\omega_\mathfrak{K}(y),
\end{align*}
and we can write
\begin{align}\label{sumBMO} 
\int_{\mathbb R^d}e^{-c\frac{\rho(x_0,y)^2}{t}}|f(y)-f_{\theta (B_t)}|d\omega_\mathfrak{K}(y)&\nonumber\\
    &\hspace{-4.5cm}\leq \left(\int_{\theta (B_t)}+\sum_{k=1}^\infty \int_{\theta (2^kB_t)\setminus\theta (2^{k-1}B_t)}\right)e^{-c\frac{\rho(x_0,y)^2}{t}}|f(y)-f_{\theta (B_t)}|d\omega_\mathfrak{K}(y)\nonumber\\
    &\hspace{-4.5cm}\leq \sum_{k=0}^\infty e^{-c2^{2k}}\int_{\theta (2^kB_t)}|f(y)-f_{\theta (B_t)}|d\omega_\mathfrak{K}(y)\\
    &\hspace{-4.5cm}\leq \sum_{k=0}^\infty e^{-c2^{2k}}\omega_\mathfrak{K}(\theta (2^kB_t))\Big(\|f\|_{{\rm BMO}^\rho (\mathbb R^d,\omega_\mathfrak{K})}+|f_{\theta (2^kB_t)}-f_{\theta (B_t)}|\Big)\nonumber\\
    &\hspace{-4.5cm}\leq \sum_{k=0}^\infty e^{-c2^{2k}}\omega_\mathfrak{K}(\theta (2^kB_t))\Big(\|f\|_{{\rm BMO}^\rho (\mathbb R^d,\omega_\mathfrak{K})}+\frac{\omega_\mathfrak K(\theta (2^kB_t)}{\omega_\mathfrak K(\theta (B_t))}\|f\|_{{\rm BMO}^\rho (\mathbb R^d,\omega_\mathfrak{K})}\Big)\nonumber\\
    &\hspace{-4.5cm}\leq C\omega_\mathfrak K(\theta (B_t))\|f\|_{{\rm BMO}^\rho (\mathbb R^d,\omega_\mathfrak{K})}\sum_{k=0}^\infty (1+2^{2kD})e^{-c2^{2k}}.\nonumber
\end{align}
Then,
$$
\int_{\mathbb R^d}T_t^\mathfrak K (x_0,y)|f(y)|d\omega_\mathfrak{K}(y)<\infty,\quad t>0.
$$
In analogous way, by using again \eqref{1.5} we can see that, for every $\ell\in \mathbb N$,
$$
\partial _t^\ell \int_{\mathbb R^d}T_t^\mathfrak K (x_0,y)f(y)d\omega_\mathfrak{K}(y)=\int_{\mathbb R^d}\partial _t^\ell T_t^\mathfrak K (x_0,y)f(y)d\omega_\mathfrak{K}(y),\quad t>0,
$$
and the last integral is absolutely convergent.

Let now $x_0\in \mathbb R^d$ and $r_0>0$. We write $B=B(x_0,r_0)$ and decompose $f=f_1+f_2+f_3$, where $f_1=(f-f_{\theta (4B)})\mathcal X_{\theta (4B)}$, $f_2=(f-f_{\theta (4B)})\mathcal X_{(\theta (4B))^{\rm c}}$ and $f_3=f_{\theta (4B)}$. Since $T_t^\mathfrak K (f_3)(x)=f_3$, $x\in \mathbb R^d$ and $t>0$, it follows that $V_\sigma (\{T_{t,m}\}_{t>0})(f_3)=0$ and 
$$
V_\sigma (\{T_{t,m}^\mathfrak K\}_{t>0})(f_2)\leq V_\sigma (\{T_{t,m}^\mathfrak K\}_{t>0})(f)+V_\sigma (\{T_{t,m}^\mathfrak K\}_{t>0})(f_1).
$$
By using \eqref{1.1}, \eqref{1.2} and that $V_\sigma (\{T_{t,m}^\mathfrak K\}_{t>0})$ is bounded on $L^2(\mathbb R^d,\omega_\mathfrak{K})$ from \cite[Proposition 7.3]{JiLi3} we obtain 
\begin{equation}\label{Vsigmaf1}
\|V_\sigma (\{T_{t,m}^\mathfrak K\}_{t>0})(f_1)\|_{L^2(\mathbb R^d,\omega_\mathfrak{K})}^2\leq C\int_{\theta (4B)}|f(y)-f_{\theta (4B)}|^2d\omega_\mathfrak{K}(y)\leq C\omega_\mathfrak{K}(B )\|f\|_{{\rm BMO}^\rho (\mathbb R^d,\omega_\mathfrak{K})}^2.
\end{equation}
Hence, $V_\sigma (\{T_{t,m}^\mathfrak K\}_{t>0})(f_1)(x)<\infty$, for almost all $x\in \mathbb R^d$. Since $V_\sigma (\{T_{t,m}^\mathfrak K\}_{t>0})(f)(x)<\infty$, for almost all $x\in \mathbb R^d$, we have that $V_\sigma (\{T_{t,m}^\mathfrak K\}_{t>0})(f_2)(x)<\infty$, for almost all $x\in \mathbb R^d$. We choose $x_1\in B$ such that $V_\sigma (\{T_{t,m}^\mathfrak K\}_{t>0})(f_2)(x_1)<\infty$.

We can write
\begin{align*}
    \frac{1}{\omega_\mathfrak{K}(B)}\int_B|V_\sigma (\{T_{t,m}^\mathfrak K\}_{t>0})(f)(x)-V_\sigma (\{T_{t,m}^\mathfrak K\}_{t>0})(f_2)(x_1)|d\omega_\mathfrak{K}(x)\\
    &\hspace{-8cm}\leq \frac{1}{\omega_\mathfrak{K}(B)}\int_B V_\sigma (\{T_{t,m}^\mathfrak K (f)(x)-T_{t,m}^\mathfrak K (f_2)(x_1)\}_{t>0})d\omega_\mathfrak{K}(x)\\
     &\hspace{-8cm}\leq \frac{1}{\omega_\mathfrak{K}(B)}\int_B V_\sigma (\{T_{t,m}^\mathfrak K\}_{t>0})(f_1)(x)d\omega_\mathfrak{K}(x)\\
     &\hspace{-8cm}\quad +\frac{1}{\omega_\mathfrak{K}(B)}\int_B V_\sigma (\{T_{t,m}^\mathfrak K (f_2)(x)-T_{t,m}^\mathfrak K (f_2)(x_1)\}_{t>0})d\omega_\mathfrak{K}(x).
\end{align*}
By applying H\"older inequality and considering \eqref{Vsigmaf1} we get 
$$
    \frac{1}{\omega_\mathfrak{K}(B)}\int_B V_\sigma (\{T_{t,m}^\mathfrak K\}_{t>0})(f_1)(x)d\omega_\mathfrak{K}(x)\leq C\|f\|_{{\rm BMO}^\rho (\mathbb R^d,\omega_\mathfrak{K})}.
$$
On the other hand, by taking into account \eqref{Vderivt} it can be seen that
$$
V_\sigma (\{T_{t,m}^\mathfrak K (f_2)(x)-T_{t,m}^\mathfrak K (f_2)(x_1)\}_{t>0})
\leq C\int_{(\theta (4B))^{\rm c}}|f_2(y)|(J_m(x,y)+J_{m+1}(x,y))d\omega_\mathfrak{K}(y),\quad x\in \mathbb R^d,
$$
where $J_0=0$ and for $r\in \mathbb N$, $r\geq 1$,
$$
J_r(x,y)=\int_0^\infty |t^{r-1}\partial _t^r(T_t^\mathfrak K (x,y)-T_t^\mathfrak K (x_1,y))|dt,\quad x\in B \mbox{ and }y\in (\theta(4B))^{\rm c}.
$$
Consider $r\in \mathbb N$, $r\geq 1$. We claim that
\begin{equation}\label{Jr}
J_r(x,y)\leq C\frac{r_0}{\rho (x_0,y)\omega_\mathfrak K(\theta (B(x_0,\rho (x_0,y))))},\quad x\in B,\;y\in (\theta (4B))^{\rm c}.
\end{equation}

Observe first that by virtue of \eqref{1.1} and \eqref{1.5} it follows that
\begin{equation}\label{c1}
|t^r\partial _t^rT_t^\mathfrak K (u,v)|\leq C\frac{e^{-c\frac{\rho(u,v)^2}{t}}}{V(u,v,\sqrt{t})}\leq C\frac{e^{-c\frac{\rho(u,v)^2}{t}}}{V(u,v,\rho (u,v))},\quad u,v\in \mathbb R^d,\;t>0,
\end{equation}
and in analogous way, by considering \eqref{1.6}, for $t>0$ and $u,v,w\in \mathbb R^d$, $|v-w|<\sqrt{t}$,
\begin{equation}\label{c2}
|t^r\partial _t^r[T_t^\mathfrak K (u,v)-T_t^\mathfrak K (u,w)]|\leq C\frac{|v-w|}{\sqrt{t}}\frac{e^{-c\frac{\rho(u,v)^2}{t}}}{V(u,v,\rho (u,v))}.
\end{equation}

Let $x\in B$ and $y\in (\theta (4B))^{\rm c}$. We decompose $J_r(x,y)$ in the following way,
$$
J_r(x,y)=\left(\int_0^{|x-x_1|^2}+\int_{|x-x_1|^2}^\infty \right)|t^{r-1}\partial _t^r(T_t^\mathfrak K (x,y)-T_t^\mathfrak K (x_1,y))|dt=J_{r,1}(x,y)+J_{r,2}(x,y).
$$
By \eqref{c1} we obtain
\begin{align}\label{Jr1}
    J_{r,1}(x,y)&\leq C\int_0^{|x-x_1|^2}\Big(\frac{e^{-c\frac{\rho(x,y)^2}{t}}}{\omega_\mathfrak{K}(B(x,\rho(x,y)))}+\frac{e^{-c\frac{\rho(x_1,y)^2}{t}}}{\omega_\mathfrak{K}(B(x_1,\rho (x_1,y)))}\Big)\frac{dt}{t}\nonumber\\
    &\leq C\left(\frac{1}{\rho(x,y)\omega_\mathfrak{K}(B(x,\rho(x,y)))}+\frac{1}{\rho(x_1,y)\omega_\mathfrak{K}(B(x_1,\rho(x_1,y)))}\right)\int_0^{|x-x_1|^2}\frac{dt}{\sqrt{t}\nonumber}\\
    &\leq C|x-x_1|\left(\frac{1}{\rho(x,y)\omega_\mathfrak{K}(B(x,\rho(x,y)))}+\frac{1}{\rho(x_1,y)\omega_\mathfrak{K}(B(x_1,\rho(x_1,y)))}\right)\nonumber\\
    &\leq C\frac{r_0}{\rho(x,y)\omega_\mathfrak{K}(B(x,\rho(x,y)))}.
\end{align}
In the last inequality we have taken into account that $\rho (x,y)\sim \rho (x_1,y)$, specifically, 
$$
\frac{1}{3}\rho (x,y)\leq \rho (x_1,y)\leq \frac{5}{3}\rho (x,y).
$$
Note that $\rho(x,y)\geq \rho(y,x_0)-\rho(x,x_0)\geq 3r_0$ and that $\rho (x,x_1)\leq |x-x_1|<2r_0$. Then $\rho(x_1,y)\geq \rho(x,y)-\rho(x,x_1)\geq \rho(x,y)-2r_0\geq \rho(x,y)-\frac{2}{3}\rho(x,y)=\frac{1}{3}\rho(x,y)$. Furthermore, $\rho (x_1,y)\leq\rho (x_1,x)+ \rho (x,y)\leq 2r_0+\rho (x,y)\leq \frac{5}{3}\rho (x,y)$. Observe also that $B(x,\rho (x,y))\subseteq B(x_1,5\rho(x_1,y))$ and thus, $\omega_\mathfrak K(B(x,\rho (x,y)))\leq C\omega_\mathfrak K(B(x_1,\rho(x_1,y)))$.

On the other hand, according to \eqref{c2} we get
\begin{align}\label{Jr2}
    J_{r,2}(x,y)&\leq C\frac{|x-x_1|}{\omega_\mathfrak{K}(B(x,\rho(x,y)))}\int_0^\infty \frac{e^{-c\frac{\rho(x,y)^2}{t}}}{t^{3/2}}dt\leq C\frac{r_0}{\rho(x,y)\omega_\mathfrak{K}(B(x,\rho(x,y)))}.
\end{align}

Observe now that $\rho(x_0,y)-r_0\leq \rho(x,y)\leq \rho (x_0,y)+r_0$. Then, $\frac{3}{4}\rho(x_0,y)\leq \rho (x,y)\leq \frac{5}{4}\rho (x_0,y)$. In addition, if $z\in B(x_0,\rho(x_0,y))$ then $|z-x|\leq \rho(x_0,y)+r_0\leq \frac{5}{4}\rho(x_0,y)\leq \frac{5}{3}\rho(x,y)$, that is, $z\in B(x, \frac{5}{3}\rho (x,y))$.

Thus, from \eqref{Jr1} and \eqref{Jr2} and considering \eqref{1.1} and \eqref{1.2} we conclude \eqref{Jr}.

By using \eqref{Jr} it follows that
\begin{align*}
\int_{(\theta (4B))^{\rm c}}|f_2(y)|J_r(x,y)d\omega_\mathfrak{K}(y)&\leq Cr_0\int_{(\theta (4B))^{\rm c}}\frac{|f_2(y)|}{\rho(x_0,y)\omega_\mathfrak{K}(\theta (B(x_0,\rho(x_0,y))))}d\omega_\mathfrak{K}(y)\\
&\leq Cr_0\sum_{k=2}^\infty \int_{\theta (2^{k+1}B)\setminus \theta(2^kB)}\frac{|f_2(y)|}{\rho(x_0,y)\omega_\mathfrak{K}(\theta(B (x_0,\rho(x_0,y))))}d\omega_\mathfrak{K}(y)\\
&\leq Cr_0\sum_{k=2}^\infty \frac{1}{2^kr_0\omega_\mathfrak{K}(\theta (2^kB))}\int_{\theta(2^{k+1}B)}|f_2(y)|d\omega_\mathfrak{K}(y)
\end{align*}
We now observe that, for every $k\in \mathbb N$, $k\geq 2$, 
\begin{align*}    \int_{\theta(2^{k+1}B)}|f_2(y)|d\omega_\mathfrak{K}(y)&\leq \omega_\mathfrak K(\theta (2^{k+1}B))\Big(\|f\|_{{\rm BMO}^\rho (\mathbb R^d,\omega_\mathfrak K)}+\sum_{r=2}^k|f_{\theta (2^{r+1}B)}-f_{\theta (2^rB)}|\Big)\\
    &\leq Ck\omega_\mathfrak K(\theta (2^{k+1}B))\|f\|_{{\rm BMO}^\rho (\mathbb R^d,\omega_\mathfrak K)},
\end{align*}
where $C$ does not depend on $k$. Thus, for every $x\in B$,
\begin{align}\label{f2Jr}
  \int_{(\theta (4B))^{\rm c}}|f_2(y)|J_r(x,y)d\omega_\mathfrak{K}(y)   &
   \leq C\|f\|_{{\rm BMO}^\rho (\mathbb R^d,\omega_\mathfrak{K})}\sum_{k=2}^\infty \frac{k}{2^k}\leq C\|f\|_{{\rm BMO}^\rho (\mathbb R^d,\omega_\mathfrak{K})}.
\end{align}
We get
$$
\frac{1}{\omega_\mathfrak{K}(B)}\int_{B}V_\sigma (\{T_{t,m}^\mathfrak K (f_2)(x)- T_{t,m}^\mathfrak K (f_2)(x_1)\}_{t>0})d\omega_\mathfrak{K}(x)\leq C\|f\|_{{\rm BMO}^\rho (\mathbb R^d,\omega_\mathfrak{K})}.
$$
By putting together the above estimates we conclude that $V_\sigma (\{T_{t,m}^\mathfrak K\}_{t>0})(f)\in {\rm BMO}(\mathbb R^d,\omega_\mathfrak{K})$ and
$$
\|V_\sigma (\{T_{t,m}^\mathfrak K\}_{t>0})(f)\|_{{\rm BMO}(\mathbb R^d,\omega_\mathfrak{K})}\leq C\|f\|_{{\rm BMO}^\rho(\mathbb R^d,\omega_\mathfrak{K})}.
$$

Our next aim is to show that $V_\sigma(\{T_{t,m}^\mathfrak K\}_{t>0})$ is bounded from ${\rm BMO}^\rho(\mathbb R^d,\omega_\mathfrak{K})$ into ${\rm BLO}(\mathbb R^d,\omega _\mathfrak K)$.

Suppose that $f \in {\rm BMO}^\rho (\mathbb R^d,\omega_\mathfrak{K})$ such that $V_\sigma(\{T_{t,m}^\mathfrak K\}_{t>0})(f)(x) <\infty$, for every $x\in \mathbb R^d \setminus A$, where $A$ has $\omega_\mathfrak{K}$-measure (equivalently, Lebesgue measure) zero. Our objective is to show that there exists $C>0$ such that, for every Euclidean ball $B$,
$$
\int_B(V_\sigma (\{T_{t,m}^\mathfrak K\}_{t>0})(f)(x)-\essinf_{y\in B}V_\sigma (\{T_{t,m}^\mathfrak K\}_{t>0})(f)(y))d\omega_\mathfrak K (x)\leq C\omega_\mathfrak K(B)\|f\|_{{\rm BMO}^\rho(\mathbb R^d,\omega_\mathfrak K)}.
$$
Note that ${\rm essinf}_{y \in B} V_\sigma(\{T_{t,m}^\mathfrak K\}_{t>0})(f)(y) < \infty$ because $V_\sigma(\{T_{t,m}^\mathfrak K\}_{t>0})(f)(x) <\infty$ for almost all $x\in \mathbb R^d$.

We are going to see that, for every $x\in \mathbb R^d$, the function $t\longrightarrow T_{t,m}^\mathfrak K (f)(x)$ is continuous in $(0,\infty)$. According to \cite[Proposition 7.3]{JiLi3} we have that, for every $x\in \mathbb R^d$ and $\delta >0$,
\begin{equation}\label{A1}
\int_{\mathbb R^d}\frac{|f(y)-f_{\theta (B(x,\delta))}|}{(\delta +\rho (x,y))\omega_\mathfrak K(B(x,\delta +\rho (x,y)))}d\omega _\mathfrak K(y)\leq \frac{C}{\delta}\|f\|_{{\rm BMO} ^\rho (\mathbb R^d, \omega _\mathfrak K)}.
\end{equation}
Let $x\in \mathbb R^d$. By using \eqref{1.1} we get
\begin{align}\label{A2}
\int_{\mathbb R^d}\frac{d\omega _\mathfrak K(y)}{(\delta +\rho (x,y))\omega_\mathfrak K(B(x,\delta +\rho (x,y)))}&\nonumber\\
&\hspace{-4cm}=\Big(\int_{B(x,1)}+\sum_{k=1}^\infty \int_{B(x,2^k)\setminus B(x,2^{k-1})}\Big)\frac{d\omega_\mathfrak K(y)}{(\delta +\rho (x,y))\omega _\mathfrak K(B(x,\delta +\rho (x,y)))}\nonumber\\
&\hspace{-4cm}\leq C\sum_{k=0}^\infty \frac{1}{\delta +2^k}\frac{\omega _\mathfrak K(B(x,2^k))}{\omega _\mathfrak K(B(x,\delta +2^{k-1}))}\leq C\sum_{k=0}^\infty \frac{1}{\delta +2^k}\Big(\frac{2^k}{\delta +2^k}\Big)^d\leq C\sum_{k=0}^\infty \frac{1}{2^k}\leq C,\quad \delta >0.
\end{align}

Let $0<a<b<\infty$. By \eqref{A1} and \eqref{A2}, there exists $h\in L^1(\mathbb R^d,\omega _\mathfrak K)$ such that
$$
\frac{|f(y)|}{(\delta +\rho (x,y))\omega_\mathfrak K(B(x,\delta +\rho (x,y)))}\leq h(y),\quad y\in \mathbb R^d\mbox{ and }\delta \in [a,b].
$$
From \eqref{1.1} and \eqref{1.5} we deduce that
\begin{align*}
|t^m\partial _t^mT_t(x,y)|&\leq C\frac{e^{-c\frac{\rho (x,y)^2}{t}}}{\omega_\mathfrak K(B(x,\sqrt{t}))}\leq \frac{C}{\omega_\mathfrak K(B(x,\sqrt{t}+\rho (x,y)))}\Big(\frac{\sqrt{t}+\rho (x,y)}{\sqrt{t}}\Big)^D\Big(\frac{\sqrt{t}}{\sqrt{t}+\rho (x,y)}\Big)^{D+1}\\
&\leq C\frac{\sqrt{t}}{(\sqrt{t}+\rho (x,y))\omega_\mathfrak K(B(x,\sqrt{t}+\rho (x,y)))},\quad x,y\in \mathbb R^d\mbox{ and }t>0.
\end{align*}
By using dominated convergence theorem we conclude that the function $t\longrightarrow T_{t,m}^\mathfrak K (f)(x)$ is continuous in $(0,\infty)$.

Let us fix $x_0 \in \mathbb R^d$ and $r_0 >0$, denote $B=B(x_0,r_0)$ and write the following decomposition $f= (f-f_{\theta (4B)}) \chi_{\theta (4B)}+ (f-f_{\theta (4B)})\chi_{(\theta (4B))^c}+ f_{\theta (4B)}=f_1+f_2+f_3$.

Let $\varepsilon >0$. For every $x\in B\setminus A$ there exist $n=n(x) \in \mathbb N$ and $\{t_j=t_j(x)\}_{j=1}^n \subset \mathbb Q$ such that $0<t_n<...<t_1$ and
\begin{eqnarray*}
V_\sigma (\{T_{t,m}^\mathfrak K\}_{t>0})(f)(x) < \Big(\sum_{j=1}^{n-1}|T_{t,m}^\mathfrak K (f)(x)_{|t=t_j} - T_{t,m}^\mathfrak K (f)(x)_{t=t_{j+1}}|^\sigma\Big)^{1/\sigma}+\varepsilon.
\end{eqnarray*}
Note that we can not assure that $\{t_j\}_{j=1}^n$ can be selected in a unique way. By proceeding as in \cite[p. 39]{BdL} we see that for every $x\in B$ we can choose $\{t_j(x)\}_{j=1}^{n(x)}$ such that the function $\mathbb V_{\sigma, n}(f)$ is measurable, where
$$
\mathbb V_{\sigma,n} (h)(x):=\Big(\sum_{j=1}^{n(x)-1} |T_{t,m} ^\mathfrak K (h)(x)_{|t=t_j(x)} -   T_{t,m}^\mathfrak K (h)(x)_{|t=t_{j+1}(x)}|^\sigma\Big)^{1/\sigma}.
$$
 Thus, for every $x\in B\setminus A$ we can write
$$
V_\sigma (\{T_{t,m}^\mathfrak K\}_{t>0})(f)(x)-\essinf_{z\in B}V_\sigma (\{T_{t,m}^\mathfrak K\}_{t>0})(f)(z)
\leq \mathbb V_{\sigma,n}(f)(x)-\essinf_{z\in B}V_\sigma (\{T_{t,m}^\mathfrak K\}_{t>0})(f)(z)+\varepsilon.
$$
We claim that
\begin{align}\label{claim}
\mathbb V_{\sigma,n} (f)(x)-\essinf_{z\in B}V_\sigma (\{T_{t,m}^\mathfrak K\}_{t>0})(f)(z)&\leq V_\sigma (\{T_{t,m}^\mathfrak K\}_{t>0})(f_1)(x)+\int_0^{4r_0^2}|\partial _tT_{t,m}^\mathfrak K(f_2)(x)|dt\nonumber\\
&\hspace{-3cm}+\sup_{u,z\in B}\int_{4r_0^2}^\infty|\partial _t[T_{t,m}^\mathfrak K (f)(u)-T_{t,m}^\mathfrak K (f)(z)]|dt,\quad x\in B\setminus A.
\end{align}

From this estimate we can proceed as follows to finish the proof. By taking in account that $V_\sigma(\{T_{t,m}^\mathfrak K\}_{t>0})$ is bounded on $L^2(\mathbb R^d,\omega_\mathfrak{K})$ and \cite[Proposition 7.3]{JiLi3} we get
$$
\int_{B} V_\sigma(\{T_{t,m}^\mathfrak K\}_{t>0})(f_1)(x) d\omega_\mathfrak{K}(x)\leq C\omega_\mathfrak{K}(B)^{1/2}\|f_1\|_{L^2(\mathbb R^d,\omega_\mathfrak K)} \leq C\omega_\mathfrak K (B)\|f\|_{{\rm {\rm BMO}}^\rho (\mathbb R^d,\omega_\mathfrak{K})}.
$$
On the other hand, note that $\partial _tT_{t,m}^\mathfrak K =t^{-1}(mT_{t,m}^\mathfrak K+T_{t,m+1}^\mathfrak K)$. Then,
\begin{align*}
 \int_0^{4r_0^2}|\partial _tT_{t,m}^\mathfrak K(f_2)(x)|dt&\leq \int_0^{4r_0^2} \big(m|T_{t,m}^\mathfrak K(f_2)(x)|+|T_{t,m+1}^\mathfrak K(f_2)(x)|\big)\frac{dt}{t},\quad x\in B.
\end{align*}
Let $r\in \mathbb N$, $r\geq 1$. By using \eqref{c1} and proceeding as in estimates \eqref{Jr1} and \eqref{f2Jr} we obtain
\begin{align*}
\int_0^{4r_0^2}|T_{t,r}^\mathfrak K(f_2)(x)|\frac{dt}{t}&\leq 
\int_{(\theta (4B))^{\rm c}} |f_2(y)|\int_0^{4r_0^2} \left|t^r\partial_t^rT_t^\mathfrak K (x,y)\right|\frac{dt}{t}d\omega_\mathfrak K(y)\\
&\leq C\int_{(\theta(4B))^{\rm c}}|f_2(y)|\int_0^{4r_0^2} \frac{e^{-c\frac{\rho(x,y)^2}{t}}}{\omega_\mathfrak{K}(B(x,\rho(x,y)))}\frac{dt}{t}d\omega_\mathfrak K(y)\\
&\leq C\|f\|_{{\rm BMO}^\rho(\mathbb R^d,\omega_\mathfrak K)},\quad x\in B,
\end{align*}
and, consequently,
$$
\int_{B}\int_0^{4r_0^2}|\partial _t(T_{t,m}^\mathfrak K(f_2)(x))|dtd\omega_\mathfrak{K}(x)\leq C\omega_\mathfrak K (B)\|f\|_{{\rm {\rm BMO}}^\rho (\mathbb R^d,\omega_\mathfrak{K})}.
$$
Consider now $u,z\in B$. As before, we can write 
$$
\int_{4r_0^2}^\infty |\partial _t[T_{t,m}^\mathfrak K (f)(u)-T_{t,m}^\mathfrak K (f)(z)]|dt\leq \int_{4r_0^2}^\infty \big[m|T_{t,m}^\mathfrak K (f)(u)-T_{t,m}^\mathfrak K (f)(z)|+|T_{t,m+1}^\mathfrak K (f)(u)-T_{t,m+1}^\mathfrak K (f)(z)|\big]\frac{dt}{t}.
$$
Let $r\in \mathbb N$, $r\geq 1$. From \eqref{1.4} we can write, for each $t>0$,
\begin{align*}
|T_{t,r}^\mathfrak K (f)(u)-T_{t,r}^\mathfrak K (f)(z)|&\leq \frac{C}{t}|T_{t,r}^\mathfrak K(f_1)(u)|+|T_{t,r}^\mathfrak K(f_1)(z)|+|T_{t,r}^\mathfrak K(f_2)(u)-T_{t,r}^\mathfrak K(f_2)(z)|.
\end{align*}
 By taking into account \eqref{c2} and arguing as in the estimates \eqref{Jr2} and \eqref{f2Jr} it follows that
\begin{equation}\label{Tf2}
 \int_{4r_0^2}^\infty |T_{t,r}^\mathfrak K(f_2)(u)-T_{t,r}^\mathfrak K(f_2)(z)|\frac{dt}{t}\leq C \|f\|_{{\rm BMO}^\rho (\mathbb R^d,\omega_\mathfrak{K})}.
\end{equation}
Also, from \eqref{1.1} and \eqref{1.5} we get
\begin{align}\label{Tf1}
\int_{4r_0^2}^\infty |T_{t,r}^\mathfrak K(f_1)(x)|\frac{dt}{t}& \leq  C\int_{\theta (4B)}|f(y)- f_{\theta (4B)}|\int_{4r_0^2}^\infty \frac{e^{-c\frac{\rho(x,y)^2}{t}}}{\omega_\mathfrak{K}(B(x,\sqrt{t}))} \frac{dt}{t}\,d\omega_\mathfrak{K}(y)\nonumber\\
&\leq \frac{C}{\omega_\mathfrak{K}(B(x,2r_0))}\int_{\theta (4B)}|f(y)- f_{\theta (4B)}|\int_{4r_0^2}^\infty \frac{\omega_\mathfrak{K}(B(x,2r_0))}{\omega_\mathfrak{K}(B(x,\sqrt{t}))} \frac{dt}{t}d\omega_\mathfrak{K}(y)\nonumber\\
&\leq C\frac{r_0^d}{\omega_\mathfrak{K}(B)}\int_{\theta (4B)}|f(y)- f_{\theta (4B)}|d\omega_\mathfrak{K}(y)\int_{4r_0^2}^\infty\frac{dt}{t^{1+\frac{d}{2}}} \nonumber\\
&\leq C \frac{1}{\omega_\mathfrak{K}(B)}\int_{\theta(4B)} |f(y)- f_{\theta (4B)}|d\omega_\mathfrak{K}(y) \leq C\|f\|_{{\rm BMO}^\rho (\mathbb R^d,\omega _\mathfrak K)},\quad x\in B.
\end{align}
Then, 
$$
\int_B\sup_{u,z\in B}\int_{4r_0^2}^\infty|\partial _t(T_{t,m}^\mathfrak K(f)(u)-T_{t,m}^\mathfrak K(f)(z))|dtd\omega_\mathfrak K(x)\leq C\omega_\mathfrak K(B)\|f\|_{{\rm BMO}^\rho(\mathbb R^d,\omega_\mathfrak K)}.
$$
By considering all these estimates and the claim \eqref{claim} we can conclude that
$$
\frac{1}{\omega_\mathfrak K(B)}\int_B\big(V_\sigma (\{T_{t,m}^\mathfrak K\}_{t>0})(f)(x)-\essinf_{z\in B}V_\sigma (\{T_{t,m}^\mathfrak K\}_{t>0})(f)(z)\big)d\omega_\mathfrak K(x)\leq C\|f\|_{{\rm BMO}^\rho(\mathbb R^d,\omega_\mathfrak K)}+\varepsilon.
$$
The arbitrariness of $\varepsilon$ leads to 
$$
\frac{1}{\omega_\mathfrak{K}(B)} \int_{B} (V_\sigma(\{T_{t,m}^\mathfrak K\}_{t>0})(f)(x) - \essinf_{z\in B} V_\sigma(\{T_{t,m}^\mathfrak K\}_{t>0})(f)(z))d\omega_\mathfrak{K}(x)\leq C\|f\|_{{\rm {\rm BMO}}^\rho (\mathbb R^d,\omega_\mathfrak{K})}.
$$
Then, $\|V_\sigma(\{T_{t,m}^\mathfrak K\}_{t>0})(f)\|_{{\rm BLO}(\mathbb R^d,\omega_\mathfrak{K})} \leq C\|f\|_{{\rm BMO}^\rho (\mathbb R^d,\omega_\mathfrak{K})}$.

In order to complete the proof we only must to verify that the claim \eqref{claim} is true. Define the following sets:  $B_1=\{x\in B\setminus A: t_1(x)<4r_0^2\}$, $B_2=\{x\in B\setminus A: 4r_0^2 \leq t_{n(x)}(x)\}$ and $B_3=\{ x\in B \setminus A: 4r_0^2\in (t_{n(x)}(x),t_1(x)]\}$. 

Suppose first that $x\in B_1$. Again by \eqref{1.4} we can write,
\begin{align}\label{B1}
\mathbb V_{\sigma,n} (f)(x)-\essinf_{z\in B}V_\sigma (\{T_{t,m}^\mathfrak K\}_{t>0})(f)(z)&\leq \mathbb V_{\sigma,n}(f_1)(x)+\mathbb V_{\sigma,n}(f_2)(x)\nonumber\\
&\hspace{-4cm}\leq V_\sigma(\{T_{t,m}^\mathfrak K\}_{t>0})(f_1)(x)+\Big(\sum_{j=1}^{n(x)-1}\Big|\int_{t_{j+1}(x)}^{t_j(x)}\partial _tT_{t,m}^\mathfrak K(f_2)(x)dt\Big|^\sigma \Big)^{1/\sigma}\nonumber\\
&\hspace{-4cm}\leq V_\sigma(\{T_{t,m}^\mathfrak K\}_{t>0})(f_1)(x)+\int_0^{4r_0^2}|\partial _tT_{t,m}^\mathfrak K(f_2)(x)|dt.
\end{align}

Consider now $x\in B_2$. In this case we have that
\begin{align}\label{B2}
\mathbb V_{\sigma,n} (f)(x)-\essinf_{z\in B}V_\sigma (\{T_{t,m}^\mathfrak K\}_{t>0})(f)(z)&\nonumber\\
&\hspace{-4.5cm}\leq \sup_{\substack{{u\in B_2}\\z \in B}}
\Big(\sum_{j=1}^{n(u)-1}\big|(T_{t,m}^\mathfrak K (f)(u)-T_{t,m}^\mathfrak K (f)(z))_{|t=t_j(u)}-(T_{t,m}^\mathfrak K (f)(u)-T_{t,m}^\mathfrak K (f)(z))_{|t=t_{j+1}(u)}\big|^\sigma\Big)^{\frac{1}{\sigma}}
\nonumber\\
&\hspace{-4.5cm}\leq \sup_{\substack{{u\in B_2}\\z \in B}}\int_{4r_0^2}^\infty |\partial _t[T_{t,m}^\mathfrak K(f)(u)-T_{t,m}^\mathfrak K(f)(z)]|dt\leq \sup_{u,z\in B}\int_{4r_0^2}^\infty |\partial _t[T_{t,m}^\mathfrak K (f)(u)-T_{t,m}^\mathfrak K (f)(z)]|dt.
\end{align}

Finally, for every $x\in B_3$ let us consider $j_0(x) \in \{2, \ldots,n(x)\}$ such that $t_{j_0(x)}(x) < 4r_0^2 \leq t_{j_0(x)-1}(x)$. Fix $x\in B_3$. It follows that
\begin{align*}
    \mathbb V_{\sigma,n}(f)(x)&\leq \Big(\sum_{j=1}^{j_0(x)-2}|T_{t,m}^\mathfrak K (f)(x)_{|t=t_j(x)}-T_{t,m}^\mathfrak K (f)(x)_{|t=t_{j+1}(x)}|^\sigma \\
    &\qquad +|T_{t,m}^\mathfrak K (f)(x)_{|t=t_{j_0(x)-1}(x)}-T_{t,m}^\mathfrak K ()f(x)_{|t=4r_0^2}|^\sigma\Big)^{1/\sigma}\\
    &\quad +\Big(\sum_{j=j_0(x)}^{n(x)-1}|T_{t,m}^\mathfrak K (f)(x)_{|t=t_j(x)}-T_{t,m}^\mathfrak K (f)(x)_{|t=t_{j+1}(x)}|^\sigma\\
    &\qquad +|T_{t,m}^\mathfrak K (f)(x)_{|t=4r_0^2}-T_{t,m}^\mathfrak K (f)(x)_{|t=t_{j_0}(x)}|^\sigma\Big)^{1/\sigma}=:\mathbb V_{\sigma,n}^1(f)(x)+\mathbb V_{\sigma,n}^2(f)(x).
\end{align*}
Observe that
\begin{align*}
\mathbb V_{\sigma,n}^1 (f)(x)-\essinf_{z\in B}V_\sigma (\{T_{t,m}^\mathfrak K\}_{t>0})(f)(z)&\\
&\hspace{-4.5cm}\leq \sup_{\substack{{u\in B_3}\\z \in B}}
\Big(\sum_{j=1}^{j_0(u)-2}\big|(T_{t,m}^\mathfrak K (f)(u)-T_{t,m}^\mathfrak K (f)(z))_{|t=t_j(u)}-(T_{t,m}^\mathfrak K (f)(u)-T_{t,m}^\mathfrak K (f)(z))_{|t=t_{j+1}(u)}\big|^\sigma\\
&\hspace{-3cm} +\big|(T_{t,m}^\mathfrak K (f)(u)-T_{t,m}^\mathfrak K (f)(z))_{|t=t_{j_0(u)-1}(u)}-(T_{t,m}^\mathfrak K (f)(u)-T_{t,m}^\mathfrak K (f)(z))_{|t=4r_0^2}\big|^\sigma\Big)^{1/\sigma}\\
&\hspace{-4.5cm}\leq \sup_{\substack{{u\in B_3}\\z \in B}}\int_{4r_0^2}^\infty  \big|\partial_t[T_{t,m}^\mathfrak K (f)(u) - T_{t,m}^\mathfrak K (f)(z)]\big|dt
\end{align*}

On the other hand,
$$
\mathbb V_{\sigma,n}^2(f)(x)\leq \mathbb V_{\sigma,n}^2(f_1)(x)+\mathbb V_{\sigma,n}^2(f_2)(x)\leq V_\sigma (\{T_{t,m}^\mathfrak K\}_{t>0})(f_1)(x)+\int_0^{4r_0^2}|\partial _tT_{t,m}^\mathfrak K(f_2)(x)|dt.
$$
We conclude that 
\begin{align}\label{B3}
\mathbb V_{\sigma,n} (f)(x)-\essinf_{z\in B}V_\sigma (\{T_{t,m}^\mathfrak K\}_{t>0})(f)(z)&\leq V_\sigma (\{T_{t,m}^\mathfrak K\}_{t>0})(f_1)(x)+\int_0^{4r_0^2}|\partial _t(T_{t,m}^\mathfrak K(f_2)(x))|dt\nonumber\\
&\hspace{-3cm}+\sup_{u,z\in B}\int_{4r_0^2}^\infty|\partial _t[T_{t,m}^\mathfrak K (f)(u)-T_{t,m}^\mathfrak K (f)(z)]|dt,\quad x\in B_3.
\end{align}
Estimates \eqref{B1}, \eqref{B2} and \eqref{B3} lead to \eqref{claim} and the proof is finished.

\section{Proof of Theorem \ref{Th1.2}}
By proceeding as in the proofs of \cite[Theorems 1.1 and 1.2]{Li} we can see that $g_m$ is bounded from $L^1(\mathbb R^d,\omega_\mathfrak{K})$ into $L^{1,\infty}(\mathbb R^d,\omega_\mathfrak{K})$. Let us establish the boundedness on $H^1(\Delta _\mathfrak K)$.
\begin{prop}
Let $m\in \mathbb N$, $m\geq 1$. The Littlewood-Paley function $g_m$ defines a bounded operator from $H^1(\Delta_\mathfrak K )$ into $L^1(\mathbb R^d,\omega_\mathfrak{K})$.
\end{prop}
\begin{proof}
It is sufficient to see that there exists $C>0$ such that
$$
\|g_m (a)\|_{L^1(\mathbb R^d, \omega_\mathfrak K)}  \leq C,
$$
for every $(1,2,\Delta_\mathfrak K ,1)$-atom $a$.

Suppose that $a$ is a $(1,2,\Delta_\mathfrak K ,1)$-atom. For a certain $b \in D(\Delta_\mathfrak K )$ and a ball $B=B(x_0,r_0)$ we have that
\begin{enumerate}
\item[({\it i})] $a= \Delta_\mathfrak K b$;
\item[({\it ii})] ${\rm supp }\,(\Delta_\mathfrak K ^\ell b)\subset \theta(B)$, $\ell=0,1$;
\item[({\it iii})] $\|(r^2_0\Delta_\mathfrak K )^\ell b\|_{L^2(\mathbb R^d,\omega_\mathfrak K)} \leq r^2_0\omega_\mathfrak{K}(B)^{-1/2}$, $\ell=0,1$.
\end{enumerate}
We can write
$$
\|g_m (a)\|_{L^1(\mathbb R^d, \omega_\mathfrak K)} =\left(\int_{\theta(4B)} + \int_{(\theta(4B))^c}\right)g_m (a)(x)d\omega_\mathfrak{K}(x)= I_1(a)+I_2(a).
$$
Since $g_m $ is bounded on $L^2(\mathbb R^d,\omega_\mathfrak{K})$ we get
$$
I_1(a) \leq \|g_m(a)\|_{L^2(\mathbb R^d,\omega_\mathfrak K)}\omega_\mathfrak{K}(\theta(4B))^{1/2}
\leq C\|a\|_{L^2(\mathbb R^d,\omega_\mathfrak K)} \omega_\mathfrak{K}(B)^{1/2} \leq C,
$$
where we have used \eqref{1.1}, \eqref{1.2} and $(iii)$ for $\ell=1$. Note that $C$ does not depend on $a$.

On the other hand, since $a=\Delta _\mathfrak Kb=\partial _tT_t^\mathfrak K(b)$, by Minkowski inequality we have that
\begin{align*}
g_m (a)(x) &=\big\|t^m\partial_t^{m+1}T_t^\mathfrak K (b)(x)\big\|_{L^2((0,\infty),\frac{dt}{t})}\\
&\leq \int_{\theta(B)}|b(y)|\big\|t^m\partial_t^{m+1}T_t^\mathfrak K (x,y)\big\|_{L^2((0,\infty),\frac{dt}{t})}d\omega_\mathfrak K(y),\quad x\in \mathbb R^d.
\end{align*}
Observe that if $x\in (\theta (4B))^{\rm c}$ and $y\in \theta (B)$ then $\frac{3}{4}\rho (x,x_0)\leq \rho (x,y)\leq \frac{5}{4}\rho (x_0,x)$, and moreover, according to \eqref{1.1} and \eqref{1.2} it follows that $\omega_\mathfrak K(B(y,\rho (x,y)))\sim \omega_\mathfrak K(B(x_0,\rho (x_0,x)))$. Thus, by using \eqref{c1} we obtain
\begin{align*}
\big\|t^m\partial_t^{m+1}T_t^\mathfrak K (x,y)\big\|_{L^2((0,\infty),\frac{dt}{t})}&\leq \frac{C}{\omega_\mathfrak K(B(y,\rho (x,y)))}\left(\int_0^\infty \frac{e^{-c\frac{\rho (x,y)^2}{t}}}{t^3}dt\right)^{1/2}\\
&\hspace{-3cm}\leq \frac{C}{\rho (x,y)^2\omega_\mathfrak K(B(y,\rho (x,y)))}\leq \frac{C}{\rho (x,x_0)^2\omega_\mathfrak K(B(x_0,\rho (x,x_0)))},\quad x\in (\theta (4B))^{\rm c},\,y\in \theta (B).
\end{align*}
Again from \eqref{1.1} and \eqref{1.2},
\begin{align*}
   \int_{(\theta (4B))^{\rm c}}\frac{d\omega_\mathfrak K(x)}{\rho (x,x_0)^2\omega_\mathfrak K(B(x_0,\rho (x,x_0)))}&\leq \sum_{k=2}^\infty \int_{\theta (2^{k+1}B)\setminus \theta (2^kB)}\frac{d\omega_\mathfrak K(x)}{\rho (x,x_0)^2\omega_\mathfrak K(B(x_0,\rho (x,x_0)))}\\
   &\leq \frac{C}{r_0^2}\sum_{k=2}^\infty \frac{\omega_\mathfrak K(\theta (2^{k+1}B))}{2^{2k}\omega_\mathfrak K(2^kB)}\leq \frac{C}{r_0^2}.
\end{align*}
We deduce that
$$
I_2(a)\leq \frac{C}{r_0^2}\int_{\theta (B)}|b(y)|d\omega _\mathfrak K(y)\leq \frac{C}{r_0^2}\|b\|_{L^2(\mathbb R^d,\omega _\mathfrak K)}\omega_\mathfrak K(\theta (B))^{1/2}\leq C,
$$
and conclude the proof.
\end{proof}
The behaviour of $g_m$ on BMO is given in the next result.
\begin{prop}
Let $m\in \mathbb N$, $m\geq 1$. Suppose that $f\in {\rm BMO}^\rho (\mathbb R^d,\omega_\mathfrak{K})$ verifies that $g_m (f)(x) < \infty$ for almost all $x \in \mathbb R^d$. Then, $g_m (f) \in {\rm BLO}(\mathbb R^d,\omega_\mathfrak{K})$ and
$$
\|g_m (f)\|_{{\rm BLO}(\mathbb R^d,\omega_\mathfrak{K})} \leq C \|f\|_{{\rm BMO}^\rho(\mathbb R^d,\omega_\mathfrak{K})},
$$
where $C>0$ does not depend on $f$.
\end{prop}
\begin{proof}
It is sufficient to see that $g_m (f)^2 \in {\rm BLO}(\mathbb R^d,\omega_\mathfrak{K})$ and
\begin{equation}\label{g3}
\|g_m (f)^2\|_{{\rm BLO}(\mathbb R^d,\omega_\mathfrak{K})} \leq C\|f\|^2_{{\rm BMO}^\rho (\mathbb R^d,\omega_\mathfrak{K})},
\end{equation}
where $C>0$ does not depend on $f$.
Indeed, let $B$ be an Euclidean ball. Since $g_m (f)(x) <\infty$ for almost all $x \in B$, $ \ef_{z\in B} g_m (f)(z) < \infty$. Then, as in \cite[p. 31]{MY}, we have that
$$
g_m (f)(x) - \ef_{z\in B} g_m (f)(z) \leq \Big((g_m (f)(x))^2-\ef_{z\in B} (g_m (f)(z))^2\Big)^{1/2},\quad  x\in B.
$$
By using Jensen's inequality we obtain
\begin{align*}
\frac{1}{\omega_\mathfrak{K}(B)}\int_B \Big(g_m (f)(x)-\ef_{z\in B} g_m (f)(z)\Big)d\omega_\mathfrak{K}(x)&\\
&\hspace{-5.5cm}\leq \frac{1}{\omega_\mathfrak{K}(B)}\int_B \Big((g_m (f)(x))^2-\ef_{z\in B} (g_m (f)(z))^2\Big)^{1/2}d\omega_\mathfrak{K}(x)\\
&\hspace{-5.5cm}\leq \left(\frac{1}{\omega_\mathfrak{K}(B)}\int_B \Big((g_m (f)(x))^2-\ef_{z\in B} (g_m (f)(z))^2\Big)d\omega_\mathfrak{K}(x)\right)^{\frac{1}{2}}.
\end{align*}
From (\ref{g3}) we deduce that $\|g_m (f)\|_{{\rm BLO}(\mathbb R^d,\omega_\mathfrak{K})} \leq C\|f\|_{{\rm BMO}^\rho (\mathbb R^d, \omega_\mathfrak{K})}$, with $C$ independent of $f$.

We now prove \eqref{g3}. Let $B=B(x_0,r_0)$ where $x_0 \in \mathbb R^d$ and $r_0>0$. We define 
$$
g_{m,0}(f)(x)=\big\|T_{t,m}^\mathfrak K (f)(x)\big\|_{L^2((0,4r_0^2),\frac{dt}{t})},\quad x\in \mathbb R^d,
$$
and
$$
g_{m,\infty}(f)(x)=\big\|T_{t,m}^\mathfrak K (f)(x)\big\|_{L^2((4r_0^2,\infty),\frac{dt}{t})},\quad x\in \mathbb R^d,
$$
and decompose $f$ as $f= (f-f_{\theta (4B)})\chi_{\theta (4B)} + (f - f_{\theta (4B)})\chi_{(\theta(4B))^{\rm c}}+ f_{\theta (4B)}=f_1+f_2+f_3$.

Property \eqref{1.4} leads to $T_{t,m}^\mathfrak K(f_3)=0$ (note that $m\geq 1$), so we have that
\begin{align}\label{e1}
(g_m (f)(x))^2-\ef_{z\in B}( g_m (f)(z))^2&\leq(g_{m,0}(f)(x))^2+(g_{m,\infty}(f)(x))^2 - \ef_{z\in B}(g_{m,\infty}(f)(z))^2\nonumber\\
&\hspace{-3cm}\leq (g_{m,0}(f_1)(x))^2+(g_{m,0}(f_2)(x))^2+(g_{m,\infty}(f)(x))^2 - \ef_{z\in B}(g_{m,\infty}(f)(z))^2,\quad x\in \mathbb R^d.
\end{align}
Since $g_m $ is bounded on $L^2(\mathbb R^d,\omega_\mathfrak{K})$, by using \eqref{1.1} and \eqref{1.2} we obtain
\begin{equation}\label{e2}
    \int_B(g_{m,0}(f_1)(x))^2d\omega_\mathfrak K(x) \leq C\int_{\theta (4B)} |f(x)-f_{\theta (4B)}|^2 d\omega_\mathfrak{K}(x)\leq C\omega_\mathfrak K(B)\|f\|^2_{{\rm BMO}^\rho (\mathbb R^d,\omega_\mathfrak{K})}.
\end{equation}
On the other hand, by considering \eqref{c1} and arguing as in the estimation \eqref{f2Jr} we get
\begin{align}\label{e3}
\int_B(g_{m,0}(f_2)(x))^2d\omega_\mathfrak K(x)&=\int_0^{4r_0^2}\int_B|T_{t,m}^\mathfrak K(f_2)(x)|^2d\omega_\mathfrak K(x)\frac{dt}{t}\nonumber\\
&\leq C\int_0^{4r_0^2}\int_B\Big(\int_{(\theta (4B))^{\rm c}}|f_2(y)|\frac{e^{-c\frac{\rho(x,y)^2}{t}}}{\omega_\mathfrak K(B(x,\rho (x,y)))}d\omega _\mathfrak K(y)\Big)^2d\omega _\mathfrak K(x)\frac{dt}{t}\nonumber\\
&\leq C\int_0^{4r_0^2}\int_B\Big(\int_{(\theta (4B))^{\rm c}}\frac{|f_2(y)|}{\rho (x,y)\omega_\mathfrak K(B(x,\rho (x,y)))}d\omega _\mathfrak K(y)\Big)^2d\omega _\mathfrak K(x)dt\nonumber\\
&=Cr_0^2\int_B\Big(\int_{(\theta (4B))^{\rm c}}\frac{|f_2(y)|}{\rho (x_0,y)\omega_\mathfrak K(B(x_0,\rho (x_0,y)))}d\omega _\mathfrak K(y)\Big)^2d\omega _\mathfrak K(x)\nonumber\\
&\leq C\omega_\mathfrak K(B)\|f\|_{{\rm BMO}^\rho (\mathbb R^d,\omega_\mathfrak K)}^2.
\end{align}
Now observe that
$$
\int_B \big[(g_{m,\infty}(f)(x))^2 - \ef_{z\in B}(g_{m,\infty}(f)(z))^2\big]d\omega_\mathfrak K(x)\leq \omega_\mathfrak K(B)\sup_{u,z\in B}|(g_{m,\infty}(f)(u))^2-(g_{m,\infty }(f)(z))^2|.
$$
We are going to see that $|(g_{m,\infty}(f)(u))^2-(g_{m,\infty }(f)(z))^2|\leq C\|f\|_{{\rm BMO}^\rho (\mathbb R^d,\omega_\mathfrak{K})}^2$, $u,z\in B$. Thus we obtain
\begin{equation}\label{e4}
\int_B \big[(g_{m,\infty}(f)(x))^2 - \ef_{z\in B}(g_{m,\infty}(f)(z))^2\big]d\omega_\mathfrak K(x)\leq C\omega_\mathfrak K(B)\|f\|_{{\rm BMO}^\rho (\mathbb R^d,\omega_\mathfrak K)}^2.
\end{equation}
Let $u,z\in B$. 
We have that
\begin{align*}
    |(g_{m,\infty}(f)(u))^2-(g_{m,\infty }(f)(z))^2|&=\left|\int_{4r_0^2}^\infty (|T_{t,m}^\mathfrak K (f)(u)|^2-|T_{t,m}^\mathfrak K (f)(z)|^2)\frac{dt}{t}\right|\\
    &\leq \int_{4r_0^2}^\infty |T_{t,m}^\mathfrak K (f)(u)-T_{t,m}^\mathfrak K (f)(z)|(|T_{t,m}^\mathfrak K (f)(u)|+|T_{t,m}^\mathfrak K (f)(z)|)\frac{dt}{t}.
\end{align*}
Let us see that for each $x\in B$ and $t>4r_0^2$, $|T_{t,m}^\mathfrak K (f)(x)|\leq C\|f\|_{{\rm BMO}(\mathbb R^d,\omega _\mathfrak K)}$. Consider $x\in B$ and $t>4r_0^2$ and choose $k_0\in \mathbb N$ ($k_0\geq 3$) such that $2^{k_0-1}r_0< \sqrt{t}\leq 2^{k_0}r_0$.
By using that $\int_{\mathbb R^d}s^m\partial _s^mT_s^\mathfrak K(x,y)d\omega _\mathfrak K(y)=0$, $s>0$, and the estimates \eqref{1.5} and \eqref{c1} we get
\begin{align*}
    |T_{t,m}^\mathfrak K (f)(x)|&=|T_{t,m}^\mathfrak K(f-f_{\theta (2^{k_0}B)})|\leq C\left(\int_{\theta (2^{k_0}B)}\frac{|f(y)-f_{\theta(2^{k_0}B)}|}{\omega_\mathfrak K(B(x,\sqrt{t}))}d\omega_\mathfrak K(y)\right.\\
    &\quad +\left.\sum_{k=k_0}^\infty\int_{\theta (2^{k+1}B) \setminus \theta (2^kB)}\frac{e^{-c\frac{\rho(x,y)^2}{t}}}{\omega_\mathfrak K(B(x,\rho (x,y)))}|f(y)-f_{\theta (2^{k_0}B) }|d\omega _\mathfrak K(y)\right)
    \end{align*}
    Since $\sqrt{t}\sim 2^{k_0}r_0$ and $x\in B$, from \eqref{1.1} and  \eqref{1.2} we get
    $$
    \int_{\theta (2^{k_0}B)}\!\!\!\frac{|f(y)-f_{\theta(2^{k_0}B)}|}{\omega_\mathfrak K(B(x,\sqrt{t}))}d\omega_\mathfrak K(y)\leq \frac{C}{\omega_\mathfrak K(\theta (2^{k_0}B))}\int_{\theta (2^{k_0}B)}\!\!\!|f(y)-f_{\theta(2^{k_0}B)}|d\omega_\mathfrak K(y)\leq C\|f\|_{{\rm BMO}^\rho(\mathbb R^d,\omega_\mathfrak K)}.
    $$
On the other hand, when $y\in \theta (2^{k+1}B) \setminus \theta (2^kB)$, with $k\geq k_0$, we have that $\rho (x,y)^2/t\sim 2^{2(k-k_0)}$ and $\omega_\mathfrak K(B(x,\rho (x,y)))\sim \omega_\mathfrak K(B(x,2^kr_0))\sim\omega_\mathfrak K(2^kB)\sim\omega_\mathfrak K(\theta (2^kB))$. Thus, arguing as in \eqref{sumBMO} we obtain
    \begin{align*}
   \sum_{k=k_0}^\infty \int_{\theta (2^{k+1}B) \setminus \theta (2^kB)}\frac{e^{-c\frac{\rho(x,y)^2}{t}}}{\omega_\mathfrak K(B(x,\rho (x,y)))}|f(y)-f_{\theta (2^{k_0}B) }|d\omega _\mathfrak K(y)& \\
    &\hspace{-6.5cm}\leq C \sum_{k=k_0}^\infty \frac{e^{-c2^{2(k-k_0)}}}{\omega_\mathfrak K(\theta (2^kB))}\int_{\theta (2^{k+1}B)}|f(y)-f_{\theta (2^{k_0}B)}|d\omega _\mathfrak K(y)\leq C\|f\|_{{\rm BMO}^\rho(\mathbb R^d,\omega_\mathfrak K)}.
\end{align*}
By the above estimates and by proceeding as in \eqref{Tf2} and \eqref{Tf1} it follows that
\begin{align*}
|(g_{m,\infty}(f)(u))^2-(g_{m,\infty }(f)(z))^2|&\leq C\|f\|_{{\rm BMO}^\rho (\mathbb R^d,\omega_\mathfrak K)}\int_{4r_0^2}^\infty |T_{t,m}^\mathfrak K (f)(u)-T_{t,m}^\mathfrak K (f)(z)|\frac{dt}{t}\\
&\leq C\|f\|_{{\rm BMO}^\rho (\mathbb R^d,\omega_\mathfrak K)}^2.
\end{align*}
By combining \eqref{e1}, \eqref{e2}, \eqref{e3} and \eqref{e4} we get \eqref{g3}. Thus, the proof is finished.
\end{proof}

\section{Proof of Theorem \ref{Th1.1}}
Since $\{T_t^\mathfrak K\}_{t>0}$ is a diffusion semigroup in $L^p(\mathbb R^d,\omega_\mathfrak{K})$, $1\leq p<\infty$, according to \cite[Corollary 4.2] {LeMX2}, the operator $T_{*,k}^\mathfrak K$ is bounded from $L^p(\mathbb R^d,\omega_\mathfrak{K})$ into itself, for every $1< p<\infty$.

Let $f\in L^1(\mathbb R^d,\omega_\mathfrak{K})$. According to \eqref{1.1}, \eqref{1.5} and \eqref{c1} we get, for $x\in\mathbb R^d$ and $t>0$,
\begin{align*}
|T_{t,m}^\mathfrak K (f)(x)|& \leq C\left(\int_{\theta (B(x,\sqrt{t}))}\frac{|f(y)|}{\omega_\mathfrak{K}(B(x,\sqrt{t}))}d\omega_\mathfrak{K}(y)\right. \\
& \quad +\left. \sum_{k=0}^\infty\int_{\theta(B(x,2^{k+1}\sqrt{t}))\setminus \theta(B(x,2^{k}\sqrt{t}))}\frac{e^{-c\frac{\rho(x,y)^2}{t}}}{\omega_\mathfrak{K}(B(x,\rho (x,y)))}|f(y)|d\omega_\mathfrak{K}(y)\right) \\
& \leq C \sum_{k=-1}^\infty \frac{e^{-c2^{2k}}}{\omega_\mathfrak{K}(\theta (B(x,2^{k+1}\sqrt{t})))}\int_{\theta(B(x,2^{k+1}\sqrt{t}))}|f(y)|d\omega_\mathfrak{K}(y).
\end{align*}
Then, $T_{*,m}^\mathfrak K (f)\leq C{\mathcal M}_\rho f$, where
$$
\mathcal M_\rho f(x)=\sup_{r>0}\frac{1}{\omega_\mathfrak K(\theta (B (x,r)))}\int_{\theta (B (x,r))}|f(y)|d\omega_\mathfrak K(y),\quad x\in \mathbb R^d.
$$
As in \cite[p. 2380]{ADH} and considering that $\theta(B(x,r))=\cup_{g\in G}B(g(x),r)$, $x\in\mathbb R^d$ and $r>0$, we get
$$
T_{*,m}^\mathfrak K(f)(x)\leq C\sum_{g\in G}{\mathcal M}_{HL}f(g(x)),\quad x\in\mathbb R^d,
$$
where $\mathcal M_{HL}$ denotes the Hardy-Littlewood maximal function on the space of homogeneous type $(\mathbb R^d, |\cdot|,\omega_\mathfrak K)$.
We conclude that $T_{*,m}^\mathfrak K$ is bounded from $L^1(\mathbb R^d,\omega_\mathfrak{K})$ into $L^{1,\infty}(\mathbb R^d,\omega_\mathfrak{K})$.

Let us show now  the behaviour of $T_{*,m}^\mathfrak K$ on $H^1(\Delta _\mathfrak K)$. By taking into account Theorem \ref{Th1.3} and that $T_{*,m}^\mathfrak K (f)\leq V_\sigma(\{T_{t,m}^\mathfrak K\}_{t>0})(f)+T_{t,m}^\mathfrak K (f)_{|t=1}$, it is sufficient to prove that $\partial _t^mT_t^\mathfrak K\,_{|t=1}$ is bounded from $H^1(\Delta _\mathfrak K)$ into $L^1(\mathbb R^d,\omega_\mathfrak{K})$ to establish that $T_{*,m}^\mathfrak K$ also maps $H^1(\Delta _\mathfrak K)$ into $L^1(\mathbb R^d,\omega_\mathfrak{K})$.

We write $Sf=\partial_t^mT_t^\mathfrak K (f)_{|t=1}$. Since $S$ is bounded from $L^1(\mathbb R^d,\omega_\mathfrak{K})$ into $L^{1,\infty}(\mathbb R^d,\omega_\mathfrak{K})$, in order to prove that $S$ defines a bounded operator from $H^1(\Delta _\mathfrak K)$ into $L^1(\mathbb R^d,\omega_\mathfrak{K})$ we are going to see that there exists $C>0$ such that
$$
\|Sa\|_{L^1(\mathbb R^d, \omega_\mathfrak K)} \leq C,
$$
for every $(1,2,\Delta_\mathfrak K ,1)$-atom $a$.

Let $a$ be a $(1,2,\Delta_\mathfrak K ,1)$-atom. There exist $b\in D(\Delta_\mathfrak K )$ and an Euclidean ball $B=B(x_0,r_0)$ with $x_0\in\mathbb R^d$ and $r_0>0$ such that

\begin{enumerate}
\item[({\it i})] $a=\Delta_\mathfrak K  b;$
\item[({\it ii})] $\mbox{supp} (\Delta_\mathfrak K ^\ell b)\subset \theta (B)$, $\ell=0,1$;
\item[({\it iii})] $\|(r^2_0\Delta_\mathfrak K  )^\ell b\|_{L^2(\mathbb R^d,\omega_\mathfrak K)}\leq r_0^2\omega_\mathfrak{K}(B)^{-1/2}$, $\ell=0,1$.
\end{enumerate}
We write
$$
\|Sa\|_{L^1(\mathbb R^d, \omega_\mathfrak K)} =\Big(\int_{\theta(4B)}+\int_{(\theta(4B))^{\rm c}}\Big)|Sa(x)|d\omega_\mathfrak{K}(x)=J_1+J_2.
$$
Since $S$ is bounded on $L^2(\mathbb R^d,\omega_\mathfrak{K})$, by using ({\it iii}) and \eqref{1.2} we get
$$
J_1\leq \|Sa\|_{L^2(\mathbb R^d,\omega_\mathfrak K)}\omega_\mathfrak K(\theta (4B))^{1/2}\leq C\|a\|_{L^2(\mathbb R^d,\omega_\mathfrak K)}\omega_\mathfrak{K}(B)^{1/2}\leq C.
$$
On the other hand, since $|Sa|=|\Delta _\mathfrak K^{m+1}T_t^\mathfrak K(b)\,_{|t=1}|$, according to \cite[Lemmas 2.3 and 2.6]{Li} and proceeding as in \eqref{cLi} we obtain 
\begin{align*}
  J_2&= \int_{\theta(B)}|b(y)|\int_{(\theta(4B))^{\rm c}}|\Delta_\mathfrak K^{m+1}T_t^\mathfrak K (x,y)|_{|t=1}d\omega_\mathfrak{K}(x) d\omega_\mathfrak{K}(y)\\
  &\leq Ce^{-c r_0^2}\int_{\theta(B)}|b(y)|d\omega_\mathfrak{K}(y)\leq Ce^{-c r_0^2}\|b\|_{L^2(\mathbb R^d,\omega_\mathfrak K)}\omega_\mathfrak{K}(B)^{1/2}\leq Cr_0^2e^{-c r_0^2}\leq C.
\end{align*} 

We conclude that $\|Sa\|_{L^1(\mathbb R^d, \omega_\mathfrak K)} \leq C$, where $C>0$ does not depend on $a$.

Finally, suppose that $f\in {\rm BMO}^\rho (\mathbb R^d,\omega_\mathfrak{K})$ such that $T_{*,m}^\mathfrak K (f)(x)<\infty$, for almost all $x\in\mathbb R^d$. We are going to see that $T_{*,m}^\mathfrak K (f)\in {\rm BLO}(\mathbb R^d,\omega_\mathfrak{K})$.

Since $T_{*,m}^\mathfrak K (f)(x)<\infty$, for almost all $x\in\mathbb R^d$, ${\rm essinf}_{x\in B}T_{*,m}^\mathfrak K (f)(x)\in [0,\infty)$, for every Euclidean ball $B$ in $\mathbb R^d$.

Let $B=B(x_0,r_0)$ be an Euclidean ball. We consider the operators
$$
T_{*,m}^{\mathfrak K,0}(f)=\sup_{0<t\leq 4r_0^2}|T_{t,m}^\mathfrak K (f)|\quad \mbox{ and }\quad T_{*,m}^{\mathfrak K,\infty}(f)=\sup_{t> 4r_0^2}|T_{t,m}^\mathfrak K (f)|,
$$
and define the sets $B_0$ and $B_\infty$ as follows 
$$
B_0=\big\{x\in B:T_{*,m}^{\mathfrak K,0}(f)(x)\geq T_{*,m}^{\mathfrak K,\infty}(f)(x)\big\}\quad \mbox{ and }\quad B_\infty=\big\{x\in B:T_{*,m}^{\mathfrak K,0}(f)(x)< T_{*,m}^{\mathfrak K,\infty}(f)(x)\big\}.
$$
and write $f=(f-f_{\theta (4B)})\chi_{\theta (4B)}+(f-f_{\theta (4B)})\chi_{(\theta (4B))^{\rm c}}+f_{\theta (4B)}=f_1+f_2+f_3$.

We have that
\begin{align*}
\int_B [T_{*,m}^\mathfrak K (f)(x)-\ef_{z\in B}T_{*,m}^\mathfrak K (f)(z)]d\omega_\mathfrak{K}(x)&\\
&\hspace{-5cm}\leq \int_{B_0} (T_{*,m}^{\mathfrak K,0}(f)(x)-\ef_{z\in B}T_{4r_0^2,m}^\mathfrak K (f)(z))d\omega_\mathfrak{K}(x)+\int_{B_\infty} (T_{*,m}^{\mathfrak K,\infty} (f)(x)-\ef_{z\in B}T_{*,m}^{\mathfrak K,\infty} (f)(z))d\omega_\mathfrak{K}(x)\\
&\hspace{-5cm} \leq C\left(\int_B (T_{*,m}^{\mathfrak K,0}(f_1)(x)d\omega_\mathfrak{K}(x)+ \int_B (T_{*,m}^{\mathfrak K,0}(f_2)(x)d\omega_\mathfrak{K}(x) \right.  \\
&\hspace{-4cm}\quad + \omega_\mathfrak{K}(B) \sup_{x,z\in B}\sup_{0<t\leq 4r_0^2}|T_{t,m}^\mathfrak K(f_3)(x)-T_{4r_0^2,m}^\mathfrak K (f)(z)|\\
&\hspace{-4cm}\quad + \left. \omega_\mathfrak{K}(B) \sup_{x,z\in B}\sup_{t> 4r_0^2}|T_{t,m}^\mathfrak K (f)(x)-T_{t,m}^\mathfrak K (f)(z)|\right)=\sum_{j=1}^4 I_j.
\end{align*}

Since $T_{*,m}^\mathfrak K$ is bounded on $L^2(\mathbb R^d,\omega_\mathfrak{K})$ it follows that
$$
 I_1  \leq \omega_\mathfrak{K}(B)^{1/2}\|T_{*,m}^{\mathfrak K,0}(f_1)\|_{L^2(\mathbb R^d,\omega_\mathfrak{K})}\leq C \omega_\mathfrak{K}(B)^{1/2}\|f_1\|_{L^2(\mathbb R^d,\omega_\mathfrak K)}  \leq C\omega_\mathfrak{K}(B)\|f\|_{{\rm BMO}^\rho(\mathbb R^d,\omega_\mathfrak{K})}.
$$
By using \eqref{1.1}, \eqref{1.2} and \eqref{c1} and proceeding as in \eqref{f2Jr} we get, for every $x\in B$ and $0<t\leq 4r_0^2$, 
\begin{align}\label{tmf2}
    |T_{t,m}^\mathfrak K(f_2)(x)| & \leq C\int_{(\theta(4B))^{\rm c}}|f_2(y)|\frac{e^{-c\frac{\rho(x,y)^2}{t}}}{\omega_\mathfrak{K}(B(x,\rho(x,y)))}d\omega_\mathfrak{K}(y) \nonumber\\
    & \leq C\sqrt{t}\int_{(\theta(4B))^{\rm c}}\frac{|f_2(y)|}{\rho(x,y)\omega_\mathfrak{K}(B(x,\rho(x,y)))}d\omega_\mathfrak{K}(y)\leq C\|f\|_{{\rm BMO}^\rho (\mathbb R^d,\omega_\mathfrak{K})}.
\end{align}
Then, $I_2\leq C\omega_\mathfrak K(B)\|f\|_{{\rm BMO}^\rho (\mathbb R^d,\omega_\mathfrak{K})}$.

Now observe that, by virtue of \eqref{1.4}, for every $x,z\in B$ and $0<t\leq 4r_0^2$,
$$
T_{4r_0^2,m}^{\mathfrak K} (f)(z)-T_{t,m}^\mathfrak K(f_3)(x)=T_{4r_0^2,m}^{\mathfrak K}(f-f_3)(z)=T_{4r_0^2,m}^{\mathfrak K}(f_1)(z)+T_{4r_0^2,m}^{\mathfrak K}(f_2)(z).
$$
From (\ref{1.1}), \eqref{1.2} and \eqref{1.5} it follows that
\begin{align*}
 |T_{4r_0^2,m}^{\mathfrak K} (f_1)(z)|  &\leq C \int_{\theta(4B)}\frac{|f_1(y)|}{\omega_\mathfrak{K}(B(z,2r_0))}d\omega_\mathfrak{K}(y) \leq \frac{C}{\omega_\mathfrak{K}(\theta (4B))} \int_{\theta (4B)}|f(y)-f_{\theta (4B)}|d\omega_\mathfrak{K}(y)\\
    & \leq C\|f\|_{{\rm BMO}^\rho (\mathbb R^d,\omega_\mathfrak{K})},\quad z\in B,
\end{align*}
which, jointly \eqref{tmf2}, leads to $I_3\leq C\omega_\mathfrak K (B)\|f\|_{{\rm BMO}^\rho (\mathbb R^d,\omega_\mathfrak K)}$.

Finally, we estimate $I_4$. Consider $x,z\in B$ and $t>4r_0^2$. According again to \eqref{1.4} we can write 
$$
|T_{t,m}^\mathfrak K (f)(x)-T_{t,m}^\mathfrak K (f)(z)| \leq |T_{t,m}^\mathfrak K(f_1)(x)|+|T_{t,m}^\mathfrak K(f_1)(z)| + |T_{t,m}^\mathfrak K(f_2)(x)-T_{t,m}^\mathfrak K(f_2)(z)|.
$$
Estimates \eqref{1.1}, \eqref{1.2} and \eqref{1.5} imply that 
\begin{align*}
|T_{t,m}^\mathfrak K(f_1)(x)|&\leq C\int_{\theta(4B)}\frac{|f_1(y)|}{\omega_\mathfrak{K}(B(x,\sqrt{t}))}d\omega_\mathfrak{K}(y) \leq C \int_{\theta (4B)}\frac{|f_1(y)|}{\omega_\mathfrak{K}(B(x,2r_0))}d\omega_\mathfrak{K}(y)\\
&\leq C\|f\|_{{\rm BMO}^\rho (\mathbb R^d,\omega_\mathfrak{K})},
\end{align*}
and in the same way $|T_{t,m}^\mathfrak K(f_1)(z)|\leq C\|f\|_{{\rm BMO}^\rho (\mathbb R^d,\omega_\mathfrak K)}$.

On the other hand, by taking into account \eqref{c2} and arguing as in \eqref{f2Jr}, we get
\begin{align*}
    |T_{t,m}^\mathfrak K(f_2)(x)-T_{t,m}^\mathfrak K(f_2)(z)|&\leq C\frac{|x-z|}{\sqrt{t}}\int_{(\theta (4B))^{\rm c}}|f_2(y)|\frac{e^{-c\frac{\rho(x,y)^2}{t}}}{\omega_\mathfrak{K}(B(x,\rho (x,y)))}d\omega_\mathfrak{K}(y)\\
    &\leq Cr_0\int_{(\theta (4B))^{\rm c}}\frac{|f_2(y)|}{\rho (x,y)\omega_\mathfrak{K}(B(x,\rho (x,y)))}d\omega_\mathfrak{K}(y)\\
    &\leq  C\|f\|_{{\rm BMO}^\rho(\mathbb R^d,\omega_\mathfrak{K})}.
\end{align*}
Then, $I_4\leq C\omega_\mathfrak K(B)\|f\|_{{\rm BMO}^\rho (\mathbb R^d,\omega_\mathfrak{K})}$.

By combining the above estimates we conclude that
$$
\frac{1}{\omega_\mathfrak{K}(B)}\int_B (T_{*,m}^\mathfrak K (f)(x)-\ef_{z\in B}T_{*,m}^\mathfrak K (f)(z))d\omega_\mathfrak{K}(x)\leq C\|f\|_{{\rm BMO}^\rho(\mathbb R^d,\omega_\mathfrak{K})},
$$
where $C>0$ does not depend on $B$. Thus, $\|T_{*,m}^\mathfrak K(f)\|_{{\rm BLO}(\mathbb R^d,\omega_\mathfrak K)}\leq C\|f\|_{{\rm BMO}^\rho (\mathbb R^d,\omega _\mathfrak K)}$.

\end{document}